\documentclass[a4paper,11pt]{amsart}

\usepackage{amsmath,amssymb,amsthm}
\usepackage[top=2cm,bottom=2cm,left=2cm,right=2cm]{geometry}

\newtheorem{lem}{Lemma}
\newtheorem{prop}{Proposition}
\theoremstyle{definition}
\newtheorem{defn}{Definition}
\newtheorem{rem}{Remark}

\begin{document}

\title[Regularising transformations for complex differential equations]{Regularising transformations for complex differential equations with movable algebraic singularities}

\author{Thomas Kecker and Galina Filipuk}

\begin{abstract}
In a 1979 paper, K. Okamoto introduced the space of initial values for the six Painlev\'e equations and their associated Hamiltonian systems, showing that these define regular initial value problems at every point of an augmented phase space, a rational surface with certain exceptional divisors removed. We show that the construction of the space of initial values remains meaningful for certain classes of second-order complex differential equations, and more generally, Hamiltonian systems, where all movable singularities of all their solutions are algebraic poles (by some authors denoted the quasi-Painlev\'e property), which is a generalisation of the Painlev\'e property. The difference here is that the initial value problems obtained in the extended phase space become regular only after an additional change of dependent and independent variables. Constructing the analogue of space of initial values for these equations in this way also serves as an algorithm to single out, from a given class of equations or system of equations, those equations which are free from movable logarithmic branch points.
\end{abstract}

\maketitle


\section{Introduction}
Differential equations and their solutions in the complex plane have been studied extensively since the 19th century. A main motivation then was that new transcendental functions could be defined and studied as the solutions of certain differential equations. For example, Airy's equation, Bessel's equation, Weber-Hermite equation etc., all of which are important in mathematical physics, have solutions which cannot be expressed in terms of elementary functions. Rather, their solutions can be given e.g.\ in terms of power series expansions around a point, convergent in certain domains, defining analytic functions there, or by asymptotic series. Briot and Bouquet \cite{BriotBouquet} noted that cases where a differential equation can be integrated directly are extremely rare, and one should therefore study the properties of the solutions of a differential equation through the equation itself, as they have demonstrated for elliptic functions. All the equations mentioned above are linear differential equations, with non-constant coefficients. The singularities of their solution are {\it fixed}, i.e.\ they can occur only at those points where either one of the coefficients of the equation becomes singular or where the coefficient multiplying the highest derivative term vanishes. The fixed singularities essentially can be read off from the equation itself, and the nature of these singularities can be determined. The case is more involved for non-linear differential equations, for which singularities can develop somewhat spontaneously, depending on the initial data, and a priori the nature of the singularities cannot be determined by inspecting the differential equation. In particular, the positions of these singularities depend on the initial data prescribed for the equation. Roughly speaking, going from one solution of the equation to a different solution under a small variation in the initial data, the position of the singularities changes in a continuous fashion. Such singularities are thus called {\it movable}. For a detailed discussion and a more exact definition of movable singularities we refer to the article \cite{Murata1988} by Murata.

A main motivation for complex analysts studying differential equations is to find new mathematical functions with properties of interest to solve problems in physics and other areas of mathematics. 'Interesting' or 'good' mathematical functions for function theorists were considered to have no movable critical points. In other words, apart from a finite number of fixed singularities, all other (movable) singularities of any solution are poles. An equation of this kind is said to have the {\it Painlev\'e property}. For example, S. Kovalevskaya \cite{kovalevskaya} identified all the integrable cases of the equations of motion of a heavy top, by demanding that their complex solutions can be expressed by Laurent series expansions, i.e.\ solutions with singularities no worse than poles. Apart from the already known cases, namely the Lagrange and Euler top, she identified one further integrable case, given by certain ratios of the principle moments of inertia of the top, which is now known as the Kovalevskaya top.

P. Painlev\'e \cite{painleve1} and his pupil B. Gambier \cite{gambier1} took on the challenge of classifying second-order ordinary differential equations of the form
\begin{equation}
\label{2ndorderrational}
y'' = R(z,y,y'),
\end{equation}
$R$ a function rational in $y,y'$ with analytic coefficients, with the property now named after Painlev\'e. The result of this classification was a list of $50$ canonical types of equations, in the sense that any equation in the class can be obtained from an equation in the list of $50$ by applying a M\"obius type transformation 
\begin{equation*}
Y(Z) = \frac{a(z) y(z) + b(z)}{c(z)y(z) + d(z)}, \quad Z = \phi(z),
\end{equation*}
where $a,b,c,d$ and $\phi$ are analytic functions. Most of the equations in the list were found to be integrable in terms of formerly (at the time of Painlev\'e) known, classical functions, such as Airy functions, Hermite functions, Bessel functions, or other special functions (solutions of certain linear second-order differential equations with non-constant coefficients), elliptic functions, or by quadrature. Only six equations in the list turned out to produce essentially new analytic functions. These nonlinear equations are now known as the six Painlev\'e equations and their non-classical solutions are commonly called Painlev\'e transcendents. (For particular values of the parameters in the Painlev\'e equations these also have classical solutions, but not for generic parameters.) 

One way to detect equations with the Painlev\'e property, within a given class, is to first check whether they satisfy certain necessary criteria. Of such criteria, although not the one originally persued by Painlev\'e himself (who used the so-called $\alpha$-method), a very common one is to perform the {\it Painlev\'e test}, which checks whether the equation admits, at every point in the complex plane, certain formal Laurent series expansions. It is then still a much more difficult task to prove whether an equation, which passes the Painlev\'e test, actually possesses the Painlev\'e property. Proofs for the Painlev\'e property of all six Painlev\'e equations were given in \cite{Shimomura2003}, although earlier proofs exist in the literature, e.g.\ \cite{Hinkkanen1999}, \cite{Okamoto2001}, \cite{Steinmetz2000} or \cite{YanHu2003}, see also the book \cite{gromak}. There also exists a completely different approach of proving the Painlev\'e property making use of the so-called isomonodromy method \cite{Fokas}.

In \cite{filipukhalburd1,filipukhalburd2,filipukhalburd3}, Filipuk and Halburd apply a similar test to certain classes of second-order differential equations, but with algebraic series expansions in a fractional power of $z-z_0$,
\begin{equation}
\label{algebraicexpansion}
y(z) = \sum_{j=0}^\infty c_j (z-z_0)^{(j-j_0)/n}, \quad j_0, n \in \mathbb{N},
\end{equation}
instead of Laurent series. The test, which again relies on recursively computing the coefficients of the series expansion, gives rise to certain resonance conditions, which need to be satisfied in order for there to be no obstruction in the recurrence. Furthermore, in the papers cited above, Filipuk and Halburd prove that the conditions, within the given classes of equations, are sufficient for all movable singularities to be algebraic poles of the form (\ref{algebraicexpansion}), with the proviso that these are reachable by analytic continuation along a path of {\it finite length}. The study of this property was continued by one of the authors for other classes of second-order equations \cite{Kecker2012} and certain Hamiltonian systems \cite{Kecker2016}.

Thus, just by inspecting a non-linear differential equation, it is far from obvious to see whether it has the Painlev\'e property, or, more generally, what types of movable singularities its solutions can develop. In this article we are concerned with a method of determining, from a given equation or system of equations, what types of singularities the equation can develop. Although our point of departure are the Painlev\'e equations, we are studying classes of differential equations and Hamiltonian systems which admit different types of movable singularities other than poles, such as algebraic poles and logarithmic singularities. Employing a method originating in algebraic geometry, called a {\it blow-up}, we will resolve certain points of indeterminacy, or {\it base points}, which an equivalent system of equations acquires in an augmented (compact) phase space which includes the points at infinity in the space of dependent variables. We will see that this method essentially gives an algorithmic procedure of determining the possible types of singularities an equation can develop and to give conditions for which certain types of singularities, in particular logarithmic singularities, cannot occur. This can therefore be seen as an alternative to the Painlev\'e test and its generalisation to algebraic series expansions.

In a 1979 paper \cite{Okamoto1979}, K. Okamoto introduced the {\it space of initial values} for each of the Painlev\'e equations. These are extended phase spaces, every point of which defines a regular intial value problem in some coordinate chart of the space for one of the Painlev\'e equations. The space of initial values is obtained by first compactifying the phase space $\mathbb{C}^2$ of $(y,y')$ to some rational surface, such as e.g. $\mathbb{P}^2$ or $\mathbb{P}^1 \times \mathbb{P}^1$ (Okamoto himself started from so-called Hirzebruch surfaces), and then applying a number of blow-ups to resolve certain points of indeterminacy the equivalent system acquires in this augmented space. A blow-up is one of the most fundamental type of bi-rational transformation. In this way, we obtain bi-rational coordinate transformations between the original dependent variable $y$ and its derivative and coordinates covering the points at infinity in which the equation is regular. The space of initial values for a Painlev\'e equation is uniformly foliated by its solutions.

Through the space of initial values, every Painlev\'e equation is thus assigned a geometric meaning. Sakai \cite{Sakai2001} classified rational elliptic surfaces by $9$-point configurations in $\mathbb{P}^2$, which correspond to the geometries of the spaces of initial values of all known discrete and differential Painlev\'e equations. In this picture, it is more appropriate to divide the Painlev\'e differential equations into $8$ different types, as the geometry of some of the Painlev\'e equations is different for certain choices of parameters. This work has been elaborated on also in the extensive article \cite{Kajiwara2017}.

In the present article, we are mainly concerned with equations and systems of equations that are not of Painlev\'e type, but for which it can be shown that all movable singularities of their solutions are algebraic poles, such as studied by Shimomura \cite{Shimomura2007, Shimomura2008}, Filipuk and Halburd \cite{filipukhalburd1,filipukhalburd2, filipukhalburd3}, and one of the authors \cite{Kecker2012,Kecker2016}. Although such equations are in general not integrable, the condition on the singularities to be algebraic, rather than containing e.g.\ logarithmic branch points, guarantees some degree of regularity. After reviewing the construction of the space of initial values for the second Painlev\'e equation in the next section, we will mainly be concerned with equations of the form
\begin{equation*}
y'' = P(z,y),
\end{equation*}
$P$ a polynomial in $y$ with analytic coefficients and, more generally, Hamiltonian systems,
\begin{equation*}
H = H(z,x(z),y(z)), \quad x'(z) = \frac{\partial H}{\partial y}, \quad y'(z) = -\frac{\partial H}{\partial x},
\end{equation*}
where $H(z,x,y)$ is polynomial in the last two arguments, with analytic coefficients in $z$. Extending the phase space to complex projective space $\mathbb{P}^2$, certain points at infinity, where the flow of the Hamiltonian vector field becomes indeterminate, are resolved using the method of blowing up these so-called {\it base points}. This process has to be repeated a number of times, until the indeterminacy disappears, leading to an analogue of the space of initial values in which each point defines a regular initial value problem, but possibly only after a change in the dependent and independent variables. For the equations considered in this article, the described method is a finite procedure resulting in differential systems which allows us to determine directly the local singularity structure of an equation, i.e.\ what types of movable singularities its solutions can exhibit,  without having to explicitly construct the solutions. In particular, it is possible to determine when an equation has logarithmic branch points and to give conditions under which these logarithmic singularities are absent. These are the same conditions as the resonances found by applying a Painlev\'e test, testing the system for the existence of formal Laurent series solutions in $z-z_0$, or its generalisation to multi-valued singularities, testing for formal series solutions in fractional powers of $z-z_0$. Moreover, in the case where logarithmic singularities are absent, the procedure allows us to conclude that the algebraic series obtained are the only possible type movable singularities. Namely, employing certain approximate first integrals, we can show that the exceptional lines arising from all except the last blow-up are {\it inaccessible} for the flow of the vector field and the solution, at a singularity, must pass through the exceptional curve from the last blow-up where its behaviour is completely determined.

\section{Okamoto's space of initial values for the Painlev\'e equations}
\label{sec:InitialValueSpace}
The space of initial values was originally constructed by Okamoto \cite{Okamoto1979} for each of the six Painlev\'e equations. The idea in that paper is to consider the respective Hamiltonian systems in an extended phase space that includes all points at infinity, in order to study the behaviour at their singularities. In the case of the Painlev\'e equations, the extended phase space (with a certain exceptional divisor removed) covers all possible points, including points at infinity, at which the system defines a regular initial value problem. One of the main aims of this paper is to show that this construction is also meaningful for a wider class of ordinary differential equations with singularities other than poles, in particular for equations with algebraic poles. The other main point we wish to make is that the process of constructing the space of initial values also serves as an algorithm to single out, from a given class of equations, those equations for which the solutions are free from logarithmic singularities. We will first review the process of constructing the space of initial values here for the case of the second Painlev\'e equation,
\begin{equation}
\label{P2}
P_{\!I\!I}: \quad y''(z) = 2y^3 + z y + \alpha, \quad \alpha \in \mathbb{C}.
\end{equation}
Note that this is a non-autonomous ($z$-dependent) Hamiltonian system, letting $x=y'$ and defining the Hamiltonian,
\begin{equation}
	\label{P2_ham}
H = x^2 - y^4 + 2zy + 2\alpha y.
\end{equation} 
Okamoto \cite{Okamoto1979} also considered a different Hamiltonian, $H_{\text{Ok}} = \frac{1}{2} x^2 - \left( y^2 + \frac{z}{2} \right) x - \left( \alpha + \frac{1}{2} \right) y$, which, by eliminating $x$, leads to the same equation. 

Here, instead of equation (\ref{P2}), we start in fact from the more general class of equations
\begin{equation}
\label{P2general}
y''(z) = 2y^3 + \beta(z) y + \alpha(z),
\end{equation}
where $\alpha$ and $\beta$ are analytic functions. One can easily find necessary conditions for equation (\ref{P2general}) to have the Painlev\'e property. This is the so-called {\it Painlev\'e test}, which is performed by inserting formal Laurent series solution of the form
\begin{equation*}
y(z) = \frac{c_{-1}}{z-z_0} + c_0 + c_1(z-z_0) + c_2(z-z_0)^2 + \cdots
\end{equation*}
into the equation, with $c_{-1} = \pm 1$ being the two possible types of leading-order behaviour in this case. Computing the coefficients $c_0,c_1,c_2,\dots$ recursively leads to certain obstructions for the formal Laurent series to exist. The case when these obstructions are absent is equivalent to the {\it resonance conditions} $\beta''(z) \equiv \alpha'(z) \equiv 0$. Thus, $\beta$ is at most a linear function in $z$ whereas $\alpha$ is a constant. The case when $\beta'(z) \neq 0$ essentially reduces equation (\ref{P2general}) to equation (\ref{P2}), up to a rescaling of $z$. When $\beta'(z) \equiv 0$, this is an equation with constant coefficients which can be integrated directly in terms of elliptic functions. We will see below how we can re-discover the resonance conditions using the method of {\it blowing up the base points}.

At a singularity $z_\ast$ of a solution of equation (\ref{P2general}), where $\alpha(z)$ and $\beta(z)$ are analytic, we have
\begin{equation*}
\lim_{z \to z_\ast} \max\{|y(z)|,|y'(z)|\} = \infty.
\end{equation*}
This is a consequence of the following lemma by Painlev\'e, which in turn follows from Cauchy's local existence and uniqueness theorem for analytic solutions of differential equation, see e.g. \cite{hille}.
\begin{lem}
\label{Painlemma}
Given a system of differential equations,
\begin{equation*}
\mathbf{y}' = \mathbf{F}(z,\mathbf{y}), \quad \mathbf{y} = (y_1,\dots,y_n),
\end{equation*}
suppose that $F$ is analytic in a neighbourhood of a point $(z_\ast,\boldsymbol{\eta})$, $\boldsymbol{\eta} = (\eta_1,\dots,\eta_n) \in \mathbb{C}^n$. If there exists a sequence $(z_i)_{i \in \mathbb{N}}$, $z_i \to z_\ast$ as $i \to \infty$ so that $y_j(z_i) \to \eta_j$ for all $j=1,\dots,n$, then $\mathbf{y}$ is analytic at $z_\ast$.
\end{lem}
Therefore, to analyse the behaviour of the solution at a singularity, it suggests itself to include the points at infinity of the phase space, i.e.\ the line at infinity in our case, as we will start constructing the space of initial values for equation (\ref{P2general}) by extending the phase space of the differential equation to the compact surface $\mathbb{P}^2$. We introduce coordinates on the three standard charts of $\mathbb{P}^2$,
\begin{equation}
\label{homcoords}
[1:y:x] = [u:v:1] = [V:1:U],
\end{equation}
where $y$ and $x=y'$ denote the original phase space variables, and the other two coordinate charts covering $\mathbb{P}^2$ are given by $u=\frac{1}{x}, v= \frac{y}{x}$ and $U=\frac{x}{y}, V=\frac{1}{y}$, respectively. In these coordinates, equation (\ref{P2general}) is expressed as follows:

\begin{equation}
\label{extendedP2}
\begin{aligned}
u'(z) & =-\frac{u^2 v \beta (z)+u^3 \alpha (z)+2 v^3}{u}, & \quad v'(z) & =-\frac{u^2 v^2 \beta (z)+u^3 v \alpha (z)-u^2+2 v^4}{u^2}, \\
U'(z) & =\frac{-U^2 V^2+V^3 \alpha (z)+V^2 \beta (z)+2}{V^2}, & \quad V'(z) & =-U V.
\end{aligned}
\end{equation}
The line at infinity of $\mathbb{P}^2$ is given by the set $I=\{u=0\} \cup \{V=0\}$ in these coordinates. On this line, the vector field defined by (\ref{extendedP2}) is infinite, apart from the point $\mathcal{P}_1: (u,v)=(0,0)$, where it is of the indeterminate form $\frac{0}{0}$. In the vicinity of any point of $I \setminus \{\mathcal{P}_1\}$, the vector field is also 'tangential' (having zero vertical component) to the line $I$, and intuitively can never reach $I \setminus \{\mathcal{P}_1\}$. Below, we will give a formal argument to show that the line at infinity, and subsequently the exceptional lines introduced by various blow-ups, are {\it inaccessible} for the flow of the vector field away from the base points. Therefore, approaching a singularity $z_\ast$ along a curve $\gamma$, there exists at least a sequence $(z_n)_{n \in \mathbb{N}} \subset \gamma$, $z_n \to z_\ast$, such that the corresponding sequence of points in $\mathbb{P}^2$, with coordinates $(u(z_n),v(z_n))$, $(U(z_n),V(z_n))$ in the respective charts, tends to the point $\mathcal{P}_1$. A point $P \in \mathbb{P}^2 \setminus I = \mathbb{C}^2$ cannot be a limit point of the sequence since by Lemma \ref{Painlemma} the solution would be analytic at $z_\ast$ after all. 

\subsection{Resolution of base points}
A dynamical systems can be interpreted as the flow of a vector field, an arrow at each point of possible initial values for the system. A solution of the system is visualised by drawing a curve which follows the direction of the arrows in a smooth way. However, there may exist points in the phase space from which vectors emerge or sink into from all possible directions, such as the point $\mathcal{P}_1$ in the preceding paragraph, at which the vector field is a priori ill-defined. In general, we start from a rational system of equations, defined in some coordinates $(u_i,v_i)$, 
\begin{equation*}
u_i'(z) = \frac{p_{i,1}(z,u_i,v_i)}{q_{i,1}(z,u_i,v_i)}, \quad v_i'(z) = \frac{p_{i,2}(z,u_i,v_i)}{q_{i,2}(z,u_i,v_i)},
\end{equation*}
where we assume that the polynomials $p_{i,1}$, $q_{i,1}$ and $p_{i,2},q_{i,2}$ are in reduced terms, respectively. (We will let the index $i$ start counting from $0,1,2,\dots$ in the following.) The points of indeterminacy of the vector field are the common zeros $(s,t)$ of either pair of polynomials, $p_{i,1}(z,s,t) = 0 = q_{i,1}(z,s,t)$, or $p_{i,2}(z,s,t) = 0 = q_{i,2}(z,s,t)$. These {\it base points} (which may also depend on $z$), at which the behaviour of the system is a priori unknown, can be resolved using the method of {\it blowing up}, a process familiar from algebraic geometry to resolve singularities of algebraic varieties, see e.g. \cite{Hartshorne} and the work by Hironaka \cite{Hironaka1964}. By a blow-up of a point $\mathcal{P}_{i+1}: (u_i,v_i)=(s,t) \in \mathbb{C}^2$, the phase space is extended by introducing a new projective line $\mathcal{L}_{i+1}$, the points of which are in one-to-one correspondence to the various directions emanating from the base point. The extended space after blowing up $\mathcal{P}_{i+1}$, the {\it centre of the blow-up}, is given by
\begin{equation}
\label{blowupspace}
\text{Bl}_{\mathcal{P}_{i+1}}(\mathbb{C}^2) = \left\{ ((u_i,v_i),[w_0:w_1]) \in \mathbb{C}^2 \times \mathbb{P}^1 : (u_i - s)\cdot w_1 = (v_i - t) \cdot w_0 \right\}.
\end{equation}

To express the differential system in the space obtained after the blow-up, two new coordinate charts are introduced, covering the portions of the space (\ref{blowupspace}) where $w_0=0$ and $w_1=0$, respectively. We denote these coordinates by
\begin{equation*}
\begin{aligned}
u_{i+1} &= u_i - s, \quad v_{i+1} &= \frac{v_i - t}{u_i - s}, \\
U_{i+1} &= \frac{u_i - s}{v_i - t}, \quad V_{i+1} &= v_i - t.
\end{aligned}
\end{equation*}
After each blow-up, we therefore obtain two new rational systems
\begin{equation}
\label{ithbl}
\begin{aligned}
u_{i+1}' &= \frac{p_{i+1,1}(z,u_{i+1},v_{i+1})}{q_{i+1,1}(z,u_{i+1},v_{i+1})} & \quad v_{i+1}' &= \frac{p_{i+1,2}(z,u_{i+1},v_{i+1})}{q_{i+1,2}(z,u_{i+1},v_{i+1})} \\
U_{i+1}' &= \frac{P_{i+1,1}(z,U_{i+1},V_{i+1})}{Q_{i+1,1}(z,U_{i+1},V_{i+1})} & \quad V_{i+1}' &= \frac{P_{i+1,2}(z,U_{i+1},V_{i+1})}{Q_{i+1,2}(z,U_{i+1},V_{i+1})}
\end{aligned}
\end{equation}
where we assume again that the polynomials $p_{i+1},q_{i+1}$ and $P_{i+1},Q_{i+1}$ are already in reduced terms. Here, the relation $U_{i+1} = v_{i+1}^{-1}$ holds where either coordinate is non-zero, since $[v_{i+1} : 1] = [1: U_{i+1}] = [w_0:w_1]$ are homogeneous coordinates on the complex projective line, equivalent to $\mathbb{P}^1$, introduced by the blow-up. This line, $\mathcal{L}_{i+1}$, is also called the exceptional line of the blown-up space $\text{Bl}_{\mathcal{P}_{i+1}}$. The points on $\mathcal{L}_{i+1}$ are said to be {\it infinitely near} to the point $\mathcal{P}_{i+1}$. The canonical projection to the first component
\begin{equation*}
\pi_{i+1} : \text{Bl}_{\mathcal{P}_{i+1}}(\mathbb{C}^2) \to \mathbb{C}^2, \quad ((u_i,v_i),[w_0:w_1]) \mapsto (u_i,v_i),
\end{equation*}
defines a homeomorphism 
\begin{equation*}
\pi_{i+1}: \text{Bl}_{\mathcal{P}_{i+1}}(\mathbb{C}^2) \setminus \mathcal{L}_{i+1} \to \mathbb{C}^2 \setminus \{\mathcal{P}_{i+1}\}, 
\end{equation*}
that is, away from the centre $\mathcal{P}_{i+1}$ and its pre-image $\mathcal{L}_{i+1} = \pi_{i+1}^{-1} (\mathcal{P}_{i+1})$, points in $\mathbb{C}^2$ are in one-to-one correspondence with points in $\text{Bl}_{\mathcal{P}_{i+1}}$. 

In the coordinates $(u_{i+1},v_{i+1})$, resp. $(U_{i+1},V_{i+1})$, the exceptional line $\mathcal{L}_{i+1}$ is parametrised by
\begin{equation*}
(u_{i+1},v_{i+1}) = (0,c), \quad c \in \mathbb{C} \quad \text{or} \quad (U_{i+1},V_{i+1}) = (C,0), \quad C \in \mathbb{C},
\end{equation*}
with $C = c^{-1}$ for $c \neq 0$.
After each blow-up, we denote the space $\mathcal{S}_{i+1} = \text{Bl}_{\mathcal{P}_{i+1}}(\mathcal{S}_i)$, obtained by blowing up $\mathcal{S}_i$ at $\mathcal{P}_{i+1}$, where $\mathcal{S}_0 = \mathbb{P}^2$. Later, the location of the points $\mathcal{P}_i$ becomes $z$-dependent, and we therefore denote the blown up spaces by $\mathcal{S}_i(z)$. Furthermore, we define the {\it infinity set} $\mathcal{I}_i(z) \subset \mathcal{S}_i(z)$ as the union of the set $\mathcal{I}_{i-1}(z)$ under the blow-up with $\mathcal{L}_i$, that is $\mathcal{I}_i(z) = \mathcal{I}_{i-1}'(z) \cup \mathcal{L}_i$, where $\mathcal{I}'$ denotes the {\it proper transform} of the set $\mathcal{I}$ under the blow-up. We define $\mathcal{L}_0 = I \setminus \{\mathcal{P}_1 \} \subset \mathbb{P}^2$ as the line at infinity with the initial base point removed. 

\subsection{Sequence of blow-ups}
For system (\ref{extendedP2}) we found the initial base point $\mathcal{P}_1: (u_0,v_0) := (u,v) = (0,0)$. This base point can be resolved by a sequence of blow-ups as described in the following. After each blow-up, we have to examine the two resulting systems of equations (\ref{ithbl}) for new base points arising on the exceptional line. The indeterminacies of the system after the blow-up of $\mathcal{P}_{i+1}$ arise as common zeros of either pair of equations,
\begin{equation*}
\begin{aligned}
p_{i+1,1}(z,0,v_{i+1}) &= 0 = q_{i+1,1}(z,0,v_{i+1}), \\
p_{i+1,2}(z,0,v_{i+1}) &= 0 = q_{i+1,2}(z,0,v_{i+1}),
\end{aligned}
\end{equation*}
for the first system, and
\begin{equation*}
\begin{aligned}
P_{i+1,1}(z,U_{i+1},0) &= 0 = Q_{i+1,1}(z,U_{i+1},0), \\
P_{i+1,2}(z,U_{i+1},0) &= 0 = Q_{i+1,2}(z,U_{i+1},0),
\end{aligned}
\end{equation*}
for the second system. However, any indeterminacy at $(u_{i+1},v_{i+1}) = (0,c)$, $c \neq 0$, of the first system is a base point if and only if this indeterminacy also presents itself at $(U_{i+1},V_{i+1}) = (c^{-1},0)$ in the other system, and vice versa, as otherwise the behaviour of the solution is determined. In addition, we can have base points at $(u_{i+1},v_{i+1})=(0,0)$ or $(U_{i+1},V_{i+1})=(0,0)$, which are only visible in one of the charts. 

We now give the sequence of blow-ups for equation (\ref{P2general}) which resolves the base point, thus leading to the space of initial values. We do not write out the system of equations after each blow-up, as these expressions soon become very lengthy, and one is advised to use an appropriate computer algebra system to identify the base points in these systems and perform the blow-ups. Here, after the second blow-up, two new base points arise, thus the sequence branches into two cascades, after which we denote the subsequent coordinates with superscripts $\pm$.

\begin{equation*}
\begin{aligned}
{} & \mathcal{P}_1: (u_0,v_0) = \left(\frac{1}{x},\frac{y}{x}\right) = (0,0) \quad \leftarrow \quad \mathcal{P}_2: (U_1,V_1) = \left( \frac{1}{y}, \frac{y}{x} \right) = (0,0) \\
\leftarrow \quad & \mathcal{P}^{\pm}_3: (u_2,v_2) = \left( \frac{1}{y}, \frac{y^2}{x} \right) = (0,\pm 1) \quad \leftarrow \quad \mathcal{P}^{\pm}_4: (u_3^\pm,v_3^\pm) = \left( \frac{1}{y}, \frac{y \left(y^2 \mp x \right)}{x} \right) = (0,0) \\
\leftarrow \quad & \mathcal{P}^{\pm}_5: (u_4^\pm,v_4^\pm) = \left( \frac{1}{y}, \frac{y^2 \left(y^2 \mp x\right)}{x} \right) = \left( 0, \mp \frac{1}{2} \beta(z) \right) \\
\leftarrow \quad & \mathcal{P}^{\pm}_6: (u_5^\pm,v_5^\pm) = \left( \frac{1}{y}, \frac{y \left(2 y^4 \pm \left( x \beta (z) -2 x y^2 \right) \right)}{2x} \right) = \left( 0, \frac{1}{2} \beta'(z) \mp \alpha(z) \right) \\
\leftarrow \quad & \mathcal{P}^{\pm}_7: (U_6^\pm,V_6^\pm) = \left( \frac{2 x/y}{2 y^5 -x \beta' \pm (2 x \alpha + x y \beta -2 x y^3)}, \frac{2 y^5 -x \beta' \pm (2 x \alpha + x y \beta -2 x y^3)}{2 x} \right) = (0,0).
\end{aligned}
\end{equation*}

After the blow-up of $\mathcal{P}_6^\pm$, the differential system is of the form
\begin{equation}
\label{system6}
\begin{aligned}
u_6^{\pm}{}' &= -\frac{2}{d^\pm(z,u^\pm_6,v^\pm_6)} \\
v_6^{\pm}{}' &= \frac{2\alpha'(z) \mp \beta''(z) + p_{6,2}(z,u_6^\pm,v_6^\pm)}{u_6^\pm \cdot d^\pm(z,u^\pm_6,v^\pm_6)} \\
U_6^{\pm}{}' &= \frac{U_6^\pm (\pm 2\alpha'(z) - \beta''(z)) + P_{6,1}(z,U_6^\pm,V_6^\pm)}{V_6^\pm \cdot D^\pm(z,U^\pm_6,V^\pm_6)} \\
V_6^{\pm}{}' &= \frac{-2 + U_6^\pm (\pm 2\alpha'(z) - \beta''(z)) + P_{6,2}(z,U_6^\pm,V_6^\pm)}{V_6^\pm \cdot D^\pm(z,U^\pm_6,V^\pm_6)}
\end{aligned}
\end{equation}
where $p_{6,2}$ and $P_{6,i}$, $i=1,2$, are polynomials in their second and third arguments. Incidentally, the zero set $d^\pm(z,u_6^\pm,v_6^\pm) = 0 = D^\pm(z,U_6^\pm,V_6^\pm)$ is the set $\mathcal{I}^\pm_5{}'(z)$, the proper transform of the exceptional curves arising from the cascades of blow-ups $\mathcal{P}_1 \leftarrow \mathcal{P}_2 \leftarrow \cdots \leftarrow \mathcal{P}_5^+$ resp. $\mathcal{P}_1 \leftarrow \mathcal{P}_2 \leftarrow \cdots \leftarrow \mathcal{P}_5^-$, as well as the line $\mathcal{L}_0$,
\begin{equation*}
\begin{aligned}
d^\pm(z,u^\pm_6,v^\pm_6) & = \pm (2-2 (u^\pm_6)^3 \alpha (z)- (u^\pm_6)^2 \beta (z)) + (u^\pm_6)^3 \beta '(z) +2 (u^\pm_6)^4 v^\pm_6, \\
D^\pm(z,U^\pm_6,V^\pm_6) &= \pm (2-2 (U^\pm_6 V^\pm_6)^3 \alpha (z)- (U^\pm_6 V^\pm_6)^2 \beta (z)) + (U^\pm_6 V^\pm_6)^3 \beta '(z) +2 (U^\pm_6)^3 (V^\pm_6)^4.
\end{aligned}
\end{equation*}

\begin{rem}
\label{sequence_remark}
After each blow-up we have performed one can check that, for the resulting vector field $(u_i',v_i')$ on the exceptional curve, the $u_i'$-component is zero, whereas the $v_i'$ component becomes infinite in each point on this curve except for the base points, i.e. $\mathcal{L}_1 \setminus \{\mathcal{P}_2\}$, $\mathcal{L}_2 \setminus \{\mathcal{P}_3^+,\mathcal{P}_3^-\}$ and $\mathcal{L}^+_i \setminus \{\mathcal{P}^+_{i+1}\}$, respectively $\mathcal{L}^-_i \setminus \{\mathcal{P}^-_{i+1}\}$, for $i=3,4,5$. In a Real picture this would be understood as the vector field becoming tangent to the exceptional curve. Here, we will show through a more formal argument that the flow of the vector field cannot pass through the exceptional curve except at the base points. Namely, there exists an auxiliary function, or approximate first integral, which remains bounded at any movable singularity. For the second Painlev\'e equation, this function is known to be
\begin{equation}
\label{W_P2}
W = H + \frac{x}{y},
\end{equation}
where $H$ is the Hamiltonian (\ref{P2_ham}). In certain proofs of the Painlev\'e property for the equation this function is needed to show that actually $y \to \infty$ at a movable singularity. In the context of the space of initial values, we can use $W$ to show that the line at infinity of $\mathbb{P}^2$, and subsequently the exceptional curves introduced by the blow-ups are {\it inaccessible} for the flow of the vector field, apart from at the base points. Namely, one can check, that the logarithmic derivative $\frac{W'}{W}$ remains finite on the line at infinity and the subsequent exceptional lines introduced by the cascade of blow-ups, except at the base points, whereas $W$ itself is infinite on these lines away from the base points. We do not write out the expressions for the function $W$ in all the coordinate charts as these become rather lengthy, but we note that this can be done routinely using computer algebra. In section \ref{sec:cubicHam}, we demonstrate this process for the Hamiltonian system given there by explicitly writing out the respective functions $W$ where this is feasible. The following lemma, using a standard integral estimate, then shows that a solution cannot pass through any of the exceptional lines on which $W$ is infinite, i.e.\ they cannot be reached by analytic continuation of a solution along a finite-length curve.

\begin{lem}
	\label{log_bounded}
	Suppose a function $W(z)$ is defined in the neighbourhood $U$ of a point $z_\ast$ such that the logarithmic derivative $\frac{d}{d z} \log W = \frac{W'}{W}$ is bounded, say by $K$, on $U$. Let $\gamma \subset U$ be a finite-length curve from some point $z_0$ where $W(z_0)$ is finite and non-zero, ending in $z_\ast$. By the estimate,
	\begin{equation*}
		|\log W(z_\ast)| \leq |\log W(z_0)| + \int_{\gamma} \left| \frac{W'}{W} \right| ds \leq |\log W(z_0)| + K \cdot \text{len}(\gamma),
	\end{equation*} 
	$\log W(z_\ast)$, and hence $W(z_\ast)$, is bounded.
\end{lem}

In other words, a solution continued along a curve $\gamma \subset \mathbb{C}$, ending in a movable singularity $z_\ast$, has to approach a base point, i.e.\ there exists at least a sequence $(z_n)_{n \in \mathbb{N}} \subset \gamma$, $z_n \to z_\ast$, such that the sequence of points $(u_i(z_n),v_i(z_n))$ or $(U_i(z_n),V_i(z_n))$ tends towards one of the base points. Otherwise we would be in the situation where the solution remains entirely in the region of the phase space where the equations define a regular initial value problem, i.e.\ no singularity can develop.
\end{rem} 

The base point $\mathcal{P}^\pm_7: (U_6^\pm,V_6^\pm)=(0,0)$ in the second chart of system (\ref{system6}) is only present if the condition 
\begin{equation}
\label{P2cond}
2\alpha'(z) \mp \beta''(z) \equiv 0,
\end{equation}
is {\it not} satisfied. This point can be blown up once further, resulting in a system with no further base points. However, the solutions of the resulting system give rise to logarithmic singularities. This behaviour is already visible in the systems $(u_6^{\pm},v_6^{\pm})$: integrating the first equation of system (\ref{system6}) with initial data on the exceptional curve after the last blow-up, $u_6^\pm=0$, and inserting this into the second equations, one obtains
\begin{equation*}
u_6^\pm = \pm (z-z_0) + O((z-z_0)^2), \quad v_6^\pm = (\pm 2\alpha'(z_0) - \beta''(z_0))\log(z-z_0) +  O(z-z_0).
\end{equation*}
In case of the conditions (\ref{P2cond}) being satisfied, an additional cancellation of a factor of $u_6^\pm$ and $V_6^\pm$ occurs in the second, respectively third, equation of system (\ref{system6}), rendering this system a regular initial value problem on the exceptional curves $\mathcal{L}^\pm_6$. Also, in this case the vector field is transversal to these lines.  With initial data $(u_6^\pm(z_0),v_6^\pm(z_0)) = (0,h)$, one obtains an analytic solution
\begin{equation*}
u_6^\pm(z) = \pm (z-z_0) + O((z-z_0)^2), \quad v_6^\pm = h + O(z-z_0),
\end{equation*}
translating into a simple pole for the original variable $y$. 
The conditions (\ref{P2cond}) are exactly the resonance conditions obtained by the Painlev\'e test, combined giving $\beta''(z)=\alpha'(z)=0$. This is the case in which equation (\ref{P2general}) essentially reduces to the second Painlev\'e equation, up to a re-scaling of $z$. We denote by $\mathcal{I}_5(z) = \mathcal{I}_5^+(z) \cup \mathcal{I}_5^-(z) \subset \mathcal{S}_5(z)$ the {\it infinity set}, that is the proper transforms of the line $\mathcal{L}_0 \subset \mathbb{P}^2$ and the exceptional curves $\mathcal{L}_1$, $\mathcal{L}_2$, $\mathcal{L}_3^+$, $\mathcal{L}_3^-$, $\mathcal{L}_4^+$, $\mathcal{L}_4^-$, $\mathcal{L}_5^+$ and $\mathcal{L}_5^-$ from the first $5$ blow-ups of both cascades of base points. Then, at any point of the set $\mathcal{S}_6(z) \setminus \mathcal{I}'_5(z)$, the system (\ref{system6}) defines a regular initial value problem, which justifies the name 'space of initial values' for this set. 

Suppose now that a solution $y(z)$ of the dynamical system, defined in $\bigcup_{z \in \mathbb{C}} \mathcal{S}_6(z)$, has a movable singularity at some point $z_\ast$ and  consider a finite-length path $\gamma \subset \Omega$ with endpoint $z_\ast$, where $\Omega \subset \mathbb{C}$ is a closed neighbourhood of $z_\ast$. We denote the lifted path, i.e.\ the path that the solution along this path traces out in the (extended) phase space, by $\Gamma \subset \bigcup_{z \in \Omega} S_6(z)$. A priori $\Gamma$ can be of finite or infinite length. Let $(z_n)_{n \in \mathbb{N}} \subset \gamma$ be a sequence of points with $z_n \to z_\ast$. Since the phase space (including all the exceptional curves) is compact, there exists a subsequence $(z_{n_k})$ such that $\Gamma(z_{n_k})$ tends to a point $P_\ast \in \mathcal{S}_6(z_\ast)$. By Remark \ref{sequence_remark} and Lemma \ref{log_bounded}, we actually have $P_\ast \in \mathcal{S}_6(z_\ast) \setminus \mathcal{I}'_5(z_\ast)$. Then, by Lemma \ref{Painlemma} we can conclude that the solution, expressed in coordinates of some chart containing $P_\ast$, is analytic at the point $z_\ast$, and therefore in a neighbourhood of $z_\ast$. Thus, the solution converges to the point $P_\ast$ in this chart as $z \to z_\ast$, which corresponds to either an analytic point or a simple pole in the original variable $y$. This also excludes the possibility that $\Gamma$ has infinite length, as the curve $\Gamma$ is the analytic image of the finite-length curve $\gamma$ in this chart.
 
In summary, the procedure of blowing up the base points allows us to single out, from the class of equations (\ref{P2general}) with general coefficients, those equations for which the solutions are free from movable logarithmic singularities. Furthermore, in the absence of logarithmic singularities, the argument in the preceding paragraph essentially establishes an alternative method of proof for the Painlev\'e property of equation (\ref{P2}).

We mention that for the alternative (Okamoto) Hamiltonian $H_{\text{Ok}}$ for equation (\ref{P2}), a different sequence of base points leads to a related space of initial values. Here, there are originally two base points in $\mathbb{P}^2$, one at $(u,v)=(0,0)$, the other at $(U,V) = (0,0)$. One of them can be resolved by $3$ successive blow-ups, the other by blowing up $6$ times, the resulting resonance conditions being equivalent to the ones obtained above.
The procedure also works for the other Painlev\'e equations. For the equation $y''= 6y^2 + \alpha(z)$, $\alpha$ analytic in $z$, one finds, after compactifying the equation on $\mathbb{P}^2$ and blowing up a sequence of $9$ base points, the condition $\alpha''(z) \equiv 0$. If this condition is satisfied, the system defines a regular initial value problem on the exceptional curve from the $9$th blow-up, and the equation essentially reduces to the first Painlev\'e equation $P_I$. Moreover, in this case there is an analytic solution around each point of the space of initial values, which, in the original variable $y(z)$ corresponds to a point where the solution is either analytic or has a double pole. For detailed blow-up calculations see also the work by Duistermaat and Joshi \cite{joshi1} for the first Painlev\'e equation and Howes and Joshi \cite{joshi2} for the second Painlev\'e equation, both performed in so-called Boutroux coordinates.

\section{Differential equations with movable algebraic singularities}
In the papers \cite{Shimomura2007,Shimomura2008}, Shimomura studied certain classes of differential equations with what he called the {\it quasi-Painlev\'e property}. This is a generalisation of the Painlev\'e property in the sense that the solutions of the equations considered may have at worst {\it algebraic poles} as movable singularities. 
\begin{defn}
By an algebraic pole we denote a singularity $z_\ast$ of $y(z)$, which, in a cut neighbourhood of $z_\ast$, can be represented by a convergent Puiseux series,
\begin{equation}
\label{algebraicpole}
y(z) = \sum_{j=0}^\infty c_j (z-z_\ast)^{(j-j_0)/n}, \quad j_0,n \in \mathbb{N}.
\end{equation}
For $n=1$ this includes the notion of an ordinary pole. If the number $n$ is chosen minimal and $n>1$, we say that $y$ has an $n$th-root type algebraic pole at $z_\ast$.
\end{defn}

Shimomura proved that, for the families of equations,
\begin{equation}
\label{ShimomuraEqns}
\begin{aligned}
P_I^{(k)}:& \quad y'' = \frac{2(2k+1)}{(2k-1)^2} y^{2k} +z \quad (k \in \mathbb{N}), \\
\qquad P_{\,I\!I}^{(k)}: & \quad y'' =  \frac{k+1}{k^2} y^{2k+1} + zy + \alpha \quad (k \in \mathbb{N} \setminus \{2\}, \quad \alpha \in \mathbb{C}),
\end{aligned}
\end{equation}
the only types of movable singularities that can occur, by analytic continuation of a local solution along {\it finite-length paths}, are of the algebraic form (\ref{algebraicpole}). For $P_I^{(k)}$, 
\begin{equation}
\label{PIk}
y(z) = (z-z_\ast)^{-\frac{2}{2k-1}} - \frac{(2k-1)^2}{2(6k-1)} z_\ast (z-z_\ast)^2 + h (z-z_\ast)^{\frac{4k}{2k-1}} + \sum_{j}^\infty c_j (z-z_\ast)^{\frac{j}{2k-1}},
\end{equation}
where $h \in \mathbb{C}$ is an integration constant, and for $P_{\,I\!I}^{(k)}$,
\begin{equation}
\label{PIIk}
y(z) = \omega_k (z-z_\ast)^{-\frac{1}{k}} - \frac{k \omega_k z_\ast}{6} (z-z_\ast)^{2-\frac{1}{k}} - \frac{k^2 \alpha}{3k+1} (z-z_\ast)^2 + h (z-z_\ast)^{2+\frac{1}{k}} + \sum_{j}^\infty c_j (z-z_\ast)^{\frac{j}{k}},
\end{equation} 
where again $h$ is an integration constant and $\omega_k \in \{1,e^{i\pi / k}\}$, i.e.\ in this case there are two essentially different types of leading-order behaviour at the singularities. The proofs in \cite{Shimomura2007,Shimomura2008} for the quasi-Painlev\'e property of these equations rely on similar methods as the proofs of the Painlev\'e property for the Painlev\'e equations in \cite{Shimomura2003}. In fact, for $k=1$, the equations $P_I^{(k)}$ and $P_{\,I\!I}^{(k)}$ reduce to the first and second Painlev\'e equations, respectively. 

Already in an earlier (1953) paper, R.A. Smith considered the class of equations
\begin{equation}
\label{SmithEqn}
y''(z) + f(y) y'(z) + g(y) = h(z),
\end{equation}
where $f$ and $g$ are polynomials in $y$. He showed that, under the condition $\deg(g) < \deg(f)$, the only types of movable singularities that can occur by analytic continuation along finite-length paths are algebraic poles of the form
\begin{equation}
\label{smithexpansion}
y(z) = \sum_{j=0}^\infty c_j (z-z_0)^{(j-1)/n}, \quad n = \deg(f).
\end{equation}
Here, as in the cases of equations $P_I^{(k)}$ and $P_{\,I\!I}^{(k)}$, it is easy to verify the existence of formal series solutions of the form (\ref{smithexpansion}), (\ref{PIk}) or (\ref{PIIk}), respectively. Namely, inserting a formal series into the respective equation, one can determine the coefficients recursively without obstruction. A harder problem is to show that {\it all} movable singularities are of this form. As mentioned above, this is similar to the difference in difficulty of showing that an equation passes the Painlev\'e test and showing that the equation has the Painlev\'e property (if it has). The problem we pose is, for a given differential equation, to determine a list of possible types of movable singularities that can occur in solutions of the equation and show that these are the only ones. In the cases of the equations by Smith (\ref{SmithEqn}) and Shimomura (\ref{ShimomuraEqns}), this was shown under the proviso that paths along which we obtain a singularity through analytically continuation, are of finite length. In \cite{smith1}, Smith gave an example of a solution with a singularity not of the form (\ref{smithexpansion}), which can be obtained only by analytic continuation of a certain solution along a path of infinite length. This singularity, at which the solution behaves very differently, is an accumulation point of algebraic singularities of the form (\ref{smithexpansion}).

Departing from the works by Smith and Shimomura, Filipuk and Halburd \cite{filipukhalburd1,filipukhalburd2,filipukhalburd3} studied more general classes of differential equations with movable algebraic poles. In \cite{filipukhalburd1}, the following class of second-order equations is studied, 
\begin{equation}
\label{2ndorderPoly}
y''(z) = \sum_{n=0}^N a_n(z) y^n,
\end{equation}
where the right-hand side is a polynomial in $y$ with analytic coefficients in some domain $\Omega \subset \mathbb{C}$. After a simple transformation, this equation can be brought into the normalised form
\begin{equation}
\label{polyNormal}
y''(z) = \tilde{a}_N y^N + \sum_{n=0}^{N-2} \tilde{a}_n(z) y^n,
\end{equation}
with a conveniently chosen constant $\tilde{a}_N \in \mathbb{C}$, and where the $y^{N-1}$ term is now absent. By inserting into equation (\ref{polyNormal}) a formal series expansion of the form
\begin{equation}
\label{formalseries}
y(x) = \sum_{j=0}^\infty c_j (z-z_0)^{(j-2)/N},
\end{equation}
and recursively computing the coefficients $c_j$, one finds a necessary condition for the singularities of the solution to be algebraic. Namely, the recurrence relation is of the form
\begin{equation}
\label{Nrecurrence}
(j+N-1)(j-2N-2) c_j = P_j(c_{0},c_{1},\dots ,c_{j-1}), \quad j=1,2,\dots,
\end{equation}
where each $P_j$, $j=1,2,\dots$ is a polynomial in all the previous coefficients $c_0,\dots,c_{j-1}$. The coefficient $c_{2N+2}$ cannot be determined in this way and the recurrence relation (\ref{Nrecurrence}) is satisfied if and only if $P_{2N+2}$ is identically zero, in which case $c_{2N+2}$ is a free parameter. This {\it resonance condition}, $P_{2N+2} \equiv 0$, is necessary for the existence of the formal algebraic series solutions (\ref{formalseries}). Note that each formal series solution (\ref{formalseries}), with distinct leading-order behaviour, gives rise to one resonance condition. By expanding the coefficients $\tilde{a}_n(z)$, $n=0,\dots,N-2$ in Taylor series, one can show that the resonance conditions are equivalent to $\tilde{a}_{N-2}''(z)=0$ for even $N$, plus an additional differential relation between the coefficient functions in the case when $N$ is odd. The main result in \cite{filipukhalburd1} is that all resonance conditions being satisfied is also {\it sufficient} for all movable singularities of any solution of the equation, reachable by analytic continuation along finite length curves, to be algebraic poles of the form (\ref{formalseries}). This is essentially achieved in two steps. First, by constructing a certain auxiliary function, or approximate first integral for the equation, similar to the function $W$ in (\ref{W_P2}), which remains bounded in the vicinity of any movable singularity. Secondly, by formally constructing regular initial value problems from these bounded quantities in certain transformed variables. Regarding the second step, we show in this article how these regular initial value problems can be obtained directly by constructing the space of initial values for the equation. Although resulting in lengthy expressions, best dealt with using computer algebra, this process yields explicit equations, thus almost automating the process of finding the regular initial value problems. Although the auxiliary functions from the first step above are not required to compute the space of initial values, we will still need them to show that certain lines in this space cannot be reached by any solution.

In the following sections, we will construct the analogue of the space of initial values for some of the equations in the class (\ref{polyNormal}), namely the cases $N=4$ and $N=5$, explicitly computing the regular initial value problem at each point of this compact space, away from the exceptional divisors introduced by the blow-ups. We will need the auxiliary functions mentioned above to show that the exceptional divisors are inaccessible for the solution, using Lemma \ref{log_bounded}. To obtain a regular initial value problem, an additional change of the dependent and independent variables is necessary after the ultimate blow-up. Furthermore, with the approach in this article we can show that, for these equations and also for the Hamiltonian systems considered in Section \ref{sec:HamiltonianSystems}, all finitely reachable movable singularities are algebraic poles, i.e.\ these equations indeed have the quasi-Painlev\'e property. This is due to the fact that for these equations, blowing up the base points is a finite procedure, i.e.\ the sequence of base points terminates and the indeterminacies can be resolved completely. We will see that, in the resulting compact space, a solution approaching the singularity has a limit point somewhere on the exceptional curve after the last blow-up, where the system defines a regular initial value problem, after a change in dependent and independent variable. By Lemma \ref{Painlemma} we can conclude that there exists an analytic solution near this point, which, transformed back into the original variables, results in an algebraic pole.

The class of second-order equations (\ref{2ndorderPoly}) is contained in a wider class of polynomial Hamiltonian systems studied by one of the authors \cite{Kecker2016},
\begin{equation*}
H(z,x,y) = x^M + y^N + \sum_{0 < i N + j M < MN} \alpha_{ij}(z) x^i y^j,
\end{equation*}
where the coefficient functions $\alpha_{i,j}(z)$ are analytic in some common domain $\Omega \subset \mathbb{C}$. Also here, under a number of resonance conditions, which can be obtained either through a Painlev\'e test involving algebraic series, or through constructing the analogue of the space of initial values, the solutions of a system in this class can be shown to have only certain algebraic poles as movable singularities.

In the case where some of the resonance conditions are not satisfied, a formal algebraic series expansion with the corresponding leading-order behaviour does not exist. This can be remedied only by the introduction of logarithmic terms $\log(z-z_0)$ in the series expansions of the solutions. In this case, performing the sequence of blow-ups leads to a space in which, although the indeterminacies of the vector field defined by the equation have been resolved, the system in general does not define regular initial value problems at any point of the infinity set. With the procedure described in this article we can recover the conditions under which the respective classes of equations are free from logarithmic branch points.

\section{Second-order equation with polynomial right-hand side of degree $4$}
\label{sec:degree4}
We will now apply the procedure outlined in Section \ref{sec:InitialValueSpace} to the class of equations
\begin{equation}
\label{degree4eqn}
y''(z) = \frac{5}{2} y^4 + \alpha(z) y^2 + \beta(z) y + \gamma(z),
\end{equation}
extending the phase space of $(y,y')=(y,x)$ from $\mathbb{C}^2$ to $\mathbb{P}^2$ and resolving the base points by successive blow-ups.
The factor of $\frac{5}{2}$ is chosen for convenience here to avoid large numerical constants in the calculations, and any $y^3$ term has been transformed away. As shown in \cite{filipukhalburd1}, a necessary and sufficient condition for all singularities of this equation to be algebraic poles is $\alpha''(z) \equiv 0$, i.e.\ $\alpha$ is either a linear function in $z$ or constant. This result was obtained by introducing an auxiliary function, which in our normalisation of the equation is given by
\begin{equation}
	\label{deg4_aux}
	W = \frac{1}{2} (y')^2 - \frac{1}{2} y^5 - \frac{\alpha(z)}{3} y^3 - \frac{\beta(z)}{2} y^2 - \gamma(z) y + \left( \sum_{k=1}^3 \frac{\xi_k(z)}{y^k} \right) y',
\end{equation}
which is essentially the Hamiltonian of the equation plus corrections given in terms the functions $\xi_1,\xi_2,\xi_3$. By the process described in \cite{filipukhalburd1}, these can be computed as
\begin{equation*}
\xi_1(z) = \frac{2}{9} \alpha'(z), \quad \xi_2(z) = \beta'(z), \quad \xi_3(z) = \frac{4}{27} \alpha(z) \alpha'(z) - 2\gamma'(z),
\end{equation*}
in which case $W$ is shown to remain bounded at any movable singularity, which is established by showing that $W$ satisfies a first-order differential equation of the form 
\begin{equation*}
	W' = P(z,1/y) W + Q(z,1/y)y' + R(z,1/y),
\end{equation*}
where $P$, $Q$ and $R$ are polynomials in their last argument.

We will now recover the condition $\alpha''(z)=0$ for the existence of algebraic singularities using an appropriate cascade of blow-ups. After that, we will use the function $W$ defined in (\ref{deg4_aux}), in conjunction with Lemma \ref{log_bounded}, to show that certain exceptional curves arising from the blow-ups cannot be reached by the solution. This will allow us to conclude that the algebraic singularities are the only ones that can occur in the solutions of the equation.

To perform the blow-ups, the equation is first extended to complex projective space $\mathbb{P}^2$ by introducing homogeneous coordinates as in (\ref{homcoords}) above. The system of equations in the new coordinates is presented as follows:
\begin{align*}
u'(z) &=-\frac{2 u^2 v^2 \alpha (z)+2 u^3 v \beta (z)+2 u^4 \gamma (z)+5 v^4}{2 u^2}, \\
v'(z) &=-\frac{2 u^2 v^3 \alpha(z)+2 u^3 v^2 \beta (z)+2 u^4 v \gamma (z)-2 u^3+5 v^5}{2 u^3}, \\
U'(z) &=-\frac{2 U^2 V^3-2 V^2 \alpha (z)-2 V^3 \beta (z)-2 V^4 \gamma (z)-5}{2 V^3}, \\
V'(z) &=-U V.
\end{align*}

We see that there is an initial base point in the first chart at $(u,v) = (0,0)$. This indeterminacy can be resolved by a cascade of $14$ blow-ups, after which one finds regular initial value problems on the exceptional curve introduced by the last blow-up, but only after an additional change of dependent and independent variables. The cascade of base points is as follows:
\begin{align*}
{} & \mathcal{P}_1: (u,v) = \left(\frac{1}{x},\frac{y}{x} \right) = (0,0) \quad \leftarrow \quad \mathcal{P}_2: (U_1,V_1) =  \left( \frac{1}{y}, \frac{y}{x} \right) = (0,0) \\
\leftarrow \quad & \mathcal{P}_3: (u_2,v_2) = \left( \frac{1}{y}, \frac{y^2}{x}\right) = (0,0) \quad
\leftarrow \quad \mathcal{P}_4: (U_3,V_3) = \left( \frac{x}{y^3}, \frac{y^2}{x} \right) = (0,0) \\
\leftarrow \quad & \mathcal{P}_5: (u_4,v_4) = \left( \frac{x}{y^3} , \frac{y^5}{x^2} \right) = (0,1) \quad \leftarrow \quad \mathcal{P}_6: (u_5,v_5) = \left( \frac{x}{y^3}, \frac{y^3 \left(y^5 - x^2 \right)}{x^3} \right) = (0,0) \\
\leftarrow \quad & \mathcal{P}_7: (u_6,v_6) = \left( \frac{x}{y^3}, \frac{y^6 \left(y^5- x^2\right)}{x^4} \right) = (0,0) \quad \leftarrow \quad \mathcal{P}_8: (u_7,v_7) = \left( \frac{x}{y^3}, \frac{y^9 \left( y^5 - x^2 \right)}{x^5} \right) = (0,0) \\
\leftarrow \quad & \mathcal{P}_9: (u_8,v_8) = \left( \frac{x}{y^3}, \frac{y^{12} \left(y^5-x^2\right)}{x^6} \right) = \left( 0, -\frac{2}{3}\alpha(z) \right) \\
\leftarrow \quad & \mathcal{P}_{10}: (u_9,v_9) = \left( \frac{x}{y^3} , \frac{y^3 \left(2 x^6 \alpha (z)-3 x^2 y^{12}+3 y^{17}\right)}{3 x^7} \right) = (0,0) \\
\leftarrow \quad & \mathcal{P}_{11}: (u_{10},v_{10}) = \left( \frac{x}{y^3}, \frac{y^6 \left(2 x^6 \alpha (z)-3 x^2 y^{12}+3 y^{17}\right)}{3 x^8} \right) = (0,-\beta(z)) \\
\leftarrow \quad & \mathcal{P}_{12}: (u_{11},v_{11}) = \left( \frac{x}{y^3}, \frac{y^3 \left(3 x^8 \beta (z)+2 x^6 y^6 \alpha (z)-3 x^2 y^{18}+3 y^{23}\right)}{3 x^9} \right) = \left( 0, \frac{4}{9}\alpha'(z) \right) \\
\leftarrow \quad & \mathcal{P}_{13}: (u_{12},v_{12}) = \left( 0, \frac{4}{3}\alpha(z)^2 - 2 \gamma(z) \right) \\ & = \left( \frac{x}{y^3}, \frac{y^3 \left(6 x^6 y^9 \alpha (z)+9 x^8 y^3 \beta (z)-4 x^9 \alpha '(z)-9 x^2 y^{21}+9 y^{26}\right)}{9 x^{10}} \right)  \\
\leftarrow \quad & \mathcal{P}_{14}: (u_{13},v_{13}) = (0,2\beta'(z)) \\ & = \left( \frac{x}{y^3}, -\frac{y^3 \left(12 x^{10} \alpha (z)^2-18 x^{10} \gamma (z)+4 x^9 y^3 \alpha '(z)-6 x^6 y^{12} \alpha (z)-9 x^8 y^6 \beta
	(z)+9 x^2 y^{24}-9 y^{29}\right)}{9 x^{11}} \right).
\end{align*}
After blowing up $\mathcal{P}_{14}$, the differential system is of the form
\begin{equation}
\label{system14}
\begin{aligned}
u_{14}' &= \frac{-81 + p_{14,1}(z,u_{14},v_{14})}{2 u_{14}^2 \cdot  d(z,u_{14},v_{14})^2}, \\
v_{14}' &= \frac{-36 \alpha''(z) + p_{14,2}(z,u_{14},v_{14})}{u_{14}^3 \cdot d(z,u_{14},v_{14})^2}, \\
U_{14}' &= \frac{36 \alpha''(z) + P_{14,1}(z,U_{14},V_{14})}{U_{14} V_{14}^3 \cdot D(z,U_{14},V_{14})^2}, \\
V_{14}' &= \frac{-81 + P_{14,2}(z,U_{14},V_{14})}{2 U_{14}^3 V_{14}^2 \cdot D(z,U_{14},V_{14})^2},
\end{aligned}
\end{equation}
where $p_{14,i}$ and $P_{14,i}$, $i=1,2$ are polynomials in the variables $u_{14},v_{14}$ and $U_{14},V_{14}$, respectively, so that on the exceptional curve $L_{14}: \{u_{14}=0\} \cup \{V_{14}=0\}$, introduced by the last blow-up, we have $p_{14,i}(z,0,v_{14}) = 0 = P_{14,i}(z,U_{14},0)$. 

The zero set of the denominators $d(z,u_{14},v_{14})$ and $D(z,U_{14},V_{14})$ of (\ref{system14}) is also called the {\it exceptional divisor}, representing the set $\mathcal{I}_{13}'(z)$ in these coordinates, that is the union of the proper transforms of the exceptional curves $\mathcal{L}_1,\dots,\mathcal{L}_{13}$ introduced by the first $13$ blow-ups together with the line at infinity $\mathcal{L}_0 = I \setminus{\mathcal{P}_1} \subset \mathbb{P}^2$,
\begin{equation*}
\begin{aligned}
d &= 9+9 u_{14}^{10} v_{14}-6 u_{14}^4 \alpha +12 u_{14}^8 \alpha ^2-9 u_{14}^6 \beta -18 u_{14}^8 \gamma +4
u_{14}^7 \alpha '+18 u_{14}^9 \beta', \\
D &= 9+9 U_{14}^9 V_{14}^{10}-6 U_{14}^4 V_{14}^4 \alpha +12 U_{14}^8 V_{14}^8 \alpha ^2-9 U_{14}^6 V_{14}^6 \beta -18 U_{14}^8 V_{14}^8 \gamma +4 U_{14}^7 V_{14}^7 \alpha '+18 U_{14}^9 V_{14}^9 \beta'.
\end{aligned}
\end{equation*}
Since all the blow-ups are bi-rational transformations, one can always solve for the original coordinates, so we can give the dependence of $y$ on $u_{14},v_{14}$, as follows:

\begin{equation}
\label{coord14}
y = u_{14}^{-2} \left(1-\frac{2}{3} u_{14}^4 \alpha +u_{14}^6 \left(-\beta +u_{14} \left(\frac{4 \alpha'}{9}+\frac{1}{3} u_{14} \left(4 \alpha^2-6 \gamma +3 u_{14} \left(u_{14} v_{14}+2 \beta'\right)\right)\right)\right)\right)^{-1}.
\end{equation}
Integrating the system (\ref{system14}) when $\alpha''(z) \neq 0$ would result in logarithmic behaviour for $v_{14}$, since, to leading order,
\begin{equation*}
u_{14} = \sqrt[3]{-\frac{3}{2}}(z-z_0)^{1/3} + O \left( (z-z_0)^{2/3} \right),
\end{equation*}
and inserting this into the second equation of would result in 
\begin{equation*}
v_{14}' = \frac{8}{27} \frac{\alpha''(z)}{z-z_0} + O\left( (z-z_0)^{-2/3} \right),
\end{equation*}
from which the logarithmic behaviour $v_{14} = \frac{8}{27} \alpha''(z_0) \log(z-z_0) + O(z-z_0)$ follows. As discussed above, $\alpha''(z) \equiv 0$ is the resonance condition, where the system admits algebraic series expansions. In this case, a cancellation of one factor of $u_{14}$ resp. $V_{14}$ occurs in the second and third equation of system (\ref{system14}), which becomes
\begin{equation}
\label{afterP14blowup_res}
\begin{aligned}
u_{14}' &= \frac{-81 + p_{14,1}(z,u_{14},v_{14})}{2 u_{14}^2 \cdot  d(z,u_{14},v_{14})^2}, \\
v_{14}' &= \frac{72 \alpha(z) \alpha'(z) + 162 \gamma'(z) + \tilde{p}_{14,2}(z,u_{14},v_{14})}{u_{14}^2 \cdot d(z,u_{14},v_{14})^2},\\
U_{14}' &= \frac{72 \alpha(z) \alpha'(z) + 162 \gamma'(z) + \tilde{P}_{14,1}(z,U_{14},V_{14})}{U_{14} V_{14}^2 \cdot D(z,U_{14},V_{14})^2}, \\
V_{14}' &= \frac{-81 + P_{14,2}(z,U_{14},V_{14})}{2 U_{14}^3 V_{14}^2 \cdot D(z,U_{14},V_{14})^2},
\end{aligned}
\end{equation}
where $\tilde{p}_{14,2}$ and $\tilde{P}_{14,1}$ are polynomials in $u_{14},v_{14}$ resp. $U_{14},V_{14}$ with $\tilde{p}_{14,2}(z,0,v_{14})=0=\tilde{P}_{14,1}(z,U_{14},0)$. In this case, the vector field becomes transversal to the exceptional line $\mathcal{L}_{14} : \{u_{14}=0\} \cup \{V_{14}=0\}$, and the system can be integrated, to leading order, e.g.\ in the coordinates $u_{14},v_{14}$ as follows,
\begin{equation*}
\begin{aligned}
u_{14} &= \sqrt[3]{-\frac{3}{2}}(z-z_0)^{1/3} + O\left( (z-z_0)^{2/3} \right), \\
v_{14} &= h + \sqrt[3]{12} \left(\frac{8}{9} \alpha(z_0)\alpha'(z_0) + 2\gamma'(z_0)\right) (z-z_0)^{1/3} + O((z-z_0)^{2/3}),
\end{aligned}
\end{equation*}
where $h$ is the second integration constant (besides $z_0$). In this way, every point on the line $\mathcal{L}_{14}$ introduced by the last blow-up, parametrised by $(u_{14},v_{14})=(0,h)$, gives rise to an algebraic series solution. Denoting by $\mathcal{S}_{14}(z)$ the space obtained by blowing up the sequence of $14$ base points, which are themselves $z$-dependent, and the set $\mathcal{I}_{13}(z)$ as above, the analogue of the space of initial values can be defined as $\mathcal{S}_{14}(z) \setminus \mathcal{I}_{13}'(z)$.
Thus, away from the set $\mathcal{I}_{13}'(z)$, every point in the space we have constructed gives rise to an initial value problem with either analytic solutions or power series solutions in $(z-z_0)^{1/3}$. The latter solutions are transversal to the exceptional curve $\mathcal{L}_{14}$ from the last blow-up. 

\begin{rem}
	\label{system14remark}
	In addition to the blow-up calculations for the vector field, it is important to show that, in each step, the solution cannot pass through the exceptional curve $\mathcal{L}_i$ at any point other than the base point $\mathcal{P}_{i+1}$. This is achieved by re-expressing the auxiliary function $W$ from (\ref{deg4_aux}) in the blown-up coordinates and verifying that the logarithmic derivative $\frac{W'}{W}$ is bounded in the neighbourhood of any point on the exceptional curve $\mathcal{L}_i \setminus \{ \mathcal{P}_{i+1} \}$, whereas $W$ itself is infinite there. Lemma \ref{log_bounded} then shows that the exceptional curve is inaccessible for the flow of the vector field other than at the base point. This is ascertainment for the intuitive notion that after each blow-up, the resulting vector field is infinite on the exceptional curve $\mathcal{L}_i$ and becomes tangent to this curve $\mathcal{L}_i$ when approached away from the base point $\mathcal{P}_{i+1}$. Although we do not give the detailed (and lengthy) expressions for $\frac{d}{dz} \log(W)$ here, we note that the above mentioned properties can be checked routinely using computer algebra.
\end{rem}
We can now conclude with the statement that, in the case of the condition $\alpha''(z)=0$ being satisfied, the only singularities are algebraic.
\begin{prop}
\label{degree4prop}
The class of equations
\begin{equation*}
y'' = y^4 + (az+b) y^2 + \beta(z) y + \gamma(z),
\end{equation*}
where $\beta$ and $\gamma$ are analytic functions and $a,b \in \mathbb{C}$, has the quasi-Painlev\'e property, with cubic-root type algebraic poles.
\end{prop}
\begin{proof}
Making a change in dependent and independent variables, the system (\ref{afterP14blowup_res}) becomes
\begin{equation}
\label{system14_z}
\begin{aligned}
\frac{d z}{d u_{14}} &= \frac{2 u_{14}^2 \cdot  d(z,u_{14},v_{14})^2}{-81 + p_{14,1}(z,u_{14},v_{14})}, \\
\frac{d v_{14}}{d u_{14}} &= 2 \cdot \frac{72 \alpha(z) \alpha'(z) + 162 \gamma'(z) + \tilde{p}_{14,2}(z,u_{14},v_{14})}{-81 + p_{14,1}(z,u_{14},v_{14})}, \\
\frac{d z}{d V_{14}} &= \frac{2 U_{14}^3 V_{14}^2 \cdot D(z,U_{14},V_{14})^2}{-81 + P_{14,2}(z,U_{14},V_{14})}, \\
\frac{d U_{14}}{d V_{14}} &= 2U_{14}^2 \cdot \frac{72 \alpha(z) \alpha'(z) + 162 \gamma'(z) + \tilde{P}_{14,1}(z,U_{14},V_{14})}{-81 + P_{14,2}(z,U_{14},V_{14})},
\end{aligned}
\end{equation}
which, for initial data $(z,u_{14},v_{14})=(z_0,0,h)$ resp. $(z,U_{14},V_{14}) = (z_0,H,0)$ on the exceptional curve $\mathcal{L}_{14}$, defines a regular initial value problem of $(z,v_{14})$ in $u_{14}$ and of $(z,U_{14})$ in $V_{14}$, respectively. Let $\gamma \subset \mathbb{C}$ be a finite-length curve ending in a movable singularity $z_\ast$. The lifted curve in the phase space is denoted by $\Gamma(z)$. Let $(z_n) \subset \gamma$, $z_n \to z_\ast$ be a sequence along the curve $\gamma$. Due to the extended phase space (with all the exceptional curves) being compact, there exists a subsequence $(z_{n_k})$ such that the lifted sequence $\Gamma(z_{n_k})$ converges to a point $P_\ast \in S_{14}(z_\ast)$.  By Remark \ref{system14remark}, we actually have $P_\ast \in S_{14} \setminus I'_{13}(z_\ast)$. If $P_\ast \notin \mathcal{L}_{14}$, we would be in the situation where the original system defines a regular initial value problem, and thus would be analytic, contradicting the assumption of a singularity at $z_\ast$. Hence, we must have $P_\ast \in \mathcal{L}_{14}$. But here system (\ref{system14_z}) has an analytic solution $(z,v_{14})$ of the form
\begin{equation*}
z(u_{14}) = z_\ast - \frac{2}{3} u_{14}^3 + O(u_{14}^4), \quad v_{14}(u_{14}) = h + 2\left(\frac{8}{9} \alpha(z_\ast)\alpha'(z_\ast) + 2\gamma'(z_\ast)\right) u_{14} + O(u_{14}^2), 
\end{equation*}
or similar for $(z,U_{14})$. Inverting these power series one obtains an algebraic series expansion for $(u_{14},v_{14})$ in terms of $(z-z_\ast)^{1/3}$, which by (\ref{coord14}) corresponds to a cubic-root type algebraic pole in the original variable $y(z)$.
\end{proof}

\section{Second-order equation with polynomial right-hand side of degree $5$}
\label{sec:degree5}
As the lowest degree example of the equation $y'' = a_N y^N + \sum_{n=0}^{N-1} a_n(z) y^n$ with odd $N>3$ we consider the case $N=5$,
\begin{equation}
\label{degree5eqn}
y''(z) = 3y^5 + \alpha(z)y^3 + \beta(z)y^2 + \gamma(z)y + \delta(z),
\end{equation}
where the coefficient $a_5=3$ is chosen for computational convenience. As was shown in \cite{filipukhalburd1}, in the odd $N$ case, two resonance conditions are necessary and sufficient for the solutions of the equation to have algebraic poles as movable singularities. These can be found by inserting the formal series expansion
\begin{equation}
\label{degree5sol}
y(z) = \sum_{j=0}^\infty c_j (z-z_0)^{(j-1)/2}
\end{equation}
into equation (\ref{degree5eqn}) and computing, for each possible leading coefficient $c_0$, the obstruction in the recurrence relation (\ref{Nrecurrence}) to determine the coefficients $c_j$, $j=1,2,\dots$. In the odd $N$ case, there are two essentially different leading-order behaviours corresponding to the initial coefficients $c_{0} \in \{1,-1\}$, yielding two distinct resonances. In this case, these conditions are equivalent to $\alpha''(z) \equiv 0$ and $(\gamma(z)^2 + 4\alpha(z))'\equiv 0$. We will now show that we can recover these conditions through the construction of the analogue of the space of initial values for equation (\ref{degree5eqn}), and moreover, that the singularities of the form (\ref{degree5sol}) are the only type of movable singularity for equation (\ref{degree5eqn}).

Extending the phase space of the equation in the variables $(y,x)=(y,y')$ to $\mathbb{P}^2$ via the relations $[1:y:x] = [u:v:1] = [V:1:U]$, one finds the following systems of equations:
\begin{equation*}
\begin{aligned}
u'(z)& =-\frac{u^2 v^3 \alpha (z)+u^3 v^2 \beta (z)+u^4 v \gamma (z)+u^5 \delta (z)+3 v^5}{u^3}, \\
v'(z) &=-\frac{u^2v^4 \alpha (z)+u^3 v^3 \beta (z)+u^4 v^2 \gamma (z)+u^5 v \delta (z)-u^4+3 v^6}{u^4}, \\
U'(z) &= \frac{3-U^2 V^4+V^2 \alpha (z)+V^3 \beta (z)+V^4 \gamma (z)+V^5 \delta (z)}{V^4}, \\ 
V'(z) &= -U V.
\end{aligned}
\end{equation*}
There is a single base point in the chart $u,v$ at $(u,v)=(0,0)$. Here we describe the sequence of blow-ups needed to completely resolve this base point. Similar to the case of the second Painlev\'e equation in Section \ref{sec:InitialValueSpace}, the sequence of base points branches into two cascades after the third blow-up, so that a total of $15$ blow-ups is required. We denote coordinates and points of the subsequent blow-ups with superscripts $\pm$, the complete cascade being as follows:
\begin{equation*}
\begin{aligned}
{} & \mathcal{P}_1: (u,v) = \left(\frac{1}{x},\frac{y}{x} \right) = (0,0) \quad \leftarrow \quad \mathcal{P}_2: (U_1,V_1) = \left( \frac{1}{y}, \frac{y}{x} \right) = (0,0) \\
\leftarrow \quad & \mathcal{P}_3: (u_2,v_2) = \left( \frac{1}{y}, \frac{y^2}{x} \right) = (0,0) \quad \leftarrow \quad \mathcal{P}^{\pm}_4: (u_3,v_3) = \left( \frac{1}{y}, \frac{y^3}{x} \right) = (0, \pm 1) \\
\leftarrow \quad & \mathcal{P}^{\pm}_5: (u_4^\pm,v_4^\pm) = \left( \frac{1}{y} , \frac{y \left( y^3 \mp x \right)}{x} \right) = (0,0) \quad \leftarrow \quad \mathcal{P}^{\pm}_6: (u_5^\pm,v_5^\pm) = \left( \frac{1}{y} , \frac{y^2 \left( y^3 \mp x \right)}{x} \right) = \left(0, \mp \frac{\alpha}{4} \right) \\
\leftarrow \quad & \mathcal{P}^{\pm}_7: (u_6^\pm,v_6^\pm) = \left( \frac{1}{y} , \frac{4 y^6 \pm y (\alpha x - 4 y^2 x)}{4 x} \right) = \left( 0, \mp \frac{\beta}{3} \right) \\
\leftarrow \quad & \mathcal{P}^{\pm}_8: (u_7^\pm,v_7^\pm) = \left( \frac{1}{y}, \frac{12 y^7 \pm y \left( 3 x y \alpha + 4 x \beta - 12 x y^3 \right)}{12 x} \right) = \left( 0, \frac{1}{32} \left(4 \alpha' \pm \left( 3 \alpha^2-16 \gamma \right)\right) \right) \\
\leftarrow \quad & \mathcal{P}^{\pm}_9: (u_8^\pm,v_8^\pm) = \left( \frac{1}{y} , \frac{y \left(96 y^7 -12 x\alpha' \mp \left( 96 x y^4 -24 x y^2 \alpha +9 x \alpha^2-32 x y \beta -48 x \gamma \right) \right)}{96 x} \right) \\
& \qquad \qquad \qquad = \left(0, \frac{1}{3} \beta' \pm \left( \frac{1}{4} \alpha \beta - \delta \right) \right).
\end{aligned}
\end{equation*}
Due to the bi-rational nature of the blow-ups, the collected coordinate transformations in each cascade of blow-ups can be inverted, resulting in
\begin{equation}
\label{coord9inv}
\begin{aligned}
y &= \frac{1}{u^\pm_9}, \qquad x = y' = (u^\pm_9)^{-3} \left(1+ (u^\pm_9)^2 \left(-\frac{\alpha}{4} +u^\pm_9 \left(-\frac{\beta}{3} + u^\pm_9 \left(\frac{1}{32} \left(3 \alpha^2-16 \gamma +4 \alpha'\right) \right. \right. \right. \right. \\ & \qquad \qquad \qquad \left. \left. \left. \left. +u^\pm_9 \left(u^\pm_9 v^\pm_9+\frac{1}{12} \left(3 \alpha \beta-12 \delta+4\beta'\right)\right)\right)\right)\right)\right)^{-1}.
\end{aligned}
\end{equation}
In the coordinates after blowing up $\mathcal{P}^\pm_9$, the system is of the following form: 
\begin{equation}
\label{system9}
\begin{aligned}
u_9^\pm{}' & = \frac{-96}{u_9^\pm \cdot d^\pm(z,u_9^\pm,v_9^\pm)}, \\
v_9^\pm{}' & = \frac{\mp 12\alpha''(z) -6\alpha(z)\alpha'(z) + 48\gamma(z) + p^\pm_{9,2}(z,u_9^\pm,v_9^\pm) }{(u_9^\pm)^2 \cdot d^\pm(z,u_9^\pm,v_9^\pm)}, \\
U_9^\pm{}' & = \frac{\pm 12\alpha''(z) + 6\alpha(z)\alpha'(z) -48 \gamma'(z) + P^\pm_{9,1}(z,U_9^\pm,V_9^\pm)}{(V_9^\pm)^2 \cdot D^\pm(z,U_9^\pm,V_9^\pm)}, \\
V_9^\pm{}' & = \frac{-96-6 U_9 \left( \alpha (z) \alpha '(z)+8 \gamma '(z)-2 \alpha''(z) \right) + P^\pm_{9,2}(z,U_9^\pm,V_9^\pm)}{(U_9^\pm)^2 V_9^\pm \cdot D^\pm(z,U_9^\pm,V_9^\pm)},
\end{aligned}
\end{equation}
where $p_{9,2}$ and $P_{9,i}$, $i=1,2$ are polynomials that are zero on the exceptional curve from the last blow-up, that is, $p^\pm_{9,2}(z,0,c) = 0$ and $P^\pm_{9,i}(z,C,0) = 0$. The zero set of $d^\pm$, resp. $D^\pm$ is the exceptional divisor, i.e.\ the proper transform of the line at infinity $\mathcal{L}_0 = I \setminus \{ \mathcal{P}_1 \} \in \mathbb{P}^2$ and the exceptional curves $\mathcal{L}_1,\dots,\mathcal{L}_8^\pm$ from the blow-ups of the two cascades $\mathcal{P}_1 \leftarrow \cdots \leftarrow \mathcal{P}_8^+$ and  $\mathcal{P}_1 \leftarrow \cdots \leftarrow \mathcal{P}_8^-$, respectively:
\begin{equation*}
\begin{aligned}
d^\pm =& \pm \left( 96-24 (u^\pm_9)^2 \alpha +9 (u^\pm_9)^4 \alpha^2-32 (u^\pm_9)^3 \beta +24 (u^\pm_9)^5 \alpha \beta -48 (u^\pm_9)^4 \gamma -96 u_9^5 \delta \right) \\ & +12 (u^\pm_9)^4 \alpha'+32 (u^\pm_9)^5 \beta' +96 (u^\pm_9)^6 v^\pm_9, \\
D^\pm =& \pm \left( 96-24 (U^\pm_9 V^\pm_9)^2 \alpha+9 (U^\pm_9 V^\pm_9)^4 \alpha^2-32 (U^\pm_9 V^\pm_9)^3 \beta +24 (U^\pm_9 V^\pm_9)^5 \alpha \beta -48 (U^\pm_9 V^\pm_9)^4 \gamma \right. \\ & \left. -96 (U^\pm_9 V^\pm_9)^5 \delta \right) +12 (U^\pm_9 V^\pm_9)^4 \alpha'+32 (U^\pm_9 V^\pm_9)^5 \beta' +96 (U^\pm_9)^5 (V^\pm_9)^6. \\
\end{aligned}
\end{equation*}
Integrating the first equation in (\ref{system9}) yields 
\begin{equation*}
u_9^\pm = i\sqrt{2}(z-z_0)^{1/2} + O\left( (z-z_0) \right),
\end{equation*}
where the sign of the square root can be absorbed into the choice of branch for $(z-z_0)^{\frac{1}{2}}$. Inserting this result into the second equation of (\ref{system9}), we see that $v_9^\pm$ has a logarithmic singularity,
\begin{equation*}
v_9^\pm(z) = \frac{1}{96} \left( \pm 6 \alpha''(z) +3 \alpha(z) \alpha'(z) -24\gamma'(z) \right) \log(z-z_0) + O\left((z-z_0)^{1/2}\right),
\end{equation*}
unless the condition 
\begin{equation*}
\pm 2\alpha''(z) + \alpha(z) \alpha'(z) - 8\gamma(z) =0,
\end{equation*}
is satisfied. This condition, for both signs, amounts to the relations
\begin{equation}
\label{condn=5}
	\alpha''(z) \equiv 0, \quad \left( \alpha(z)^2 - 16\gamma(z) \right)' \equiv 0.
\end{equation}
In this case, a cancellation of one factor of $u_9^\pm$ resp. $V_9^\pm$ occurs in the second and third equation of system (\ref{system9}). Then, by changing the role of dependent and independent variables, the system is of the following form:
\begin{equation}
\label{system9regular}
\begin{aligned}
\frac{d z}{d u_9^\pm} & = -\frac{u_9^\pm \cdot d^\pm(z,u_9^\pm,v_9^\pm)}{96}, \\
\frac{d v_9^\pm}{d u_9^\pm} & = -\frac{\tilde{p}_{9,2}^\pm(z,u_9^\pm,v_9^\pm)}{96}, \\
\frac{d z}{d V^\pm_9} & = \frac{(U_9^\pm)^2 V_9^\pm \cdot D^\pm(z,U_9^\pm,V_9^\pm)}{-96 +P^\pm_{9,2}(z,U_9^\pm,V_9^\pm)}, \\
\frac{d U^\pm_9}{d V^\pm_9} & = \frac{\tilde{P}^\pm_{9,1}(z,U^\pm_9,V^\pm_9)}{-96 + P^\pm_{9,2}(z,U^\pm_9,V^\pm_9)},
\end{aligned}
\end{equation}
where $\tilde{p}^\pm_{9,2} = \frac{1}{u^\pm_9} p^\pm_{9,2}$ and $\tilde{P}^\pm_{9,1} = \frac{1}{V^\pm_9} P_{9,1}$ are polynomials in $u_{9}^\pm,v_{9}^\pm$ and $U^\pm_9,V^\pm_0$, respectively. 
For initial values $(z,u_9^\pm,v_9^\pm)=(z_0,0,h)$, respectively $(z,U^\pm_9,V^\pm_9)=(z_0,H,0)$, on the exceptional curve $\mathcal{L}_9^\pm : \{u_9^\pm = 0\} \cup \{V^\pm_9=0\}$, this defines a regular initial value problem with analytic solutions, e.g.
\begin{equation}
\label{system9solution}
\begin{aligned}
z(u_9^\pm) &= z_0 -\frac{1}{2} (u_9^\pm)^2 + O\left( (u_9^\pm)^3 \right), \\
v_9^\pm(u_9^\pm) &= h + O(u_9^\pm).
\end{aligned}
\end{equation}
Inverting these expansions we find the algebraic series solutions
\begin{equation}
\label{syst9sol}
u_9^\pm(z) = i\sqrt{2} (z-z_0)^{1/2} + O(z-z_0), \quad v_9^\pm(z) = h + O((z-z_0)^{1/2}),
\end{equation}
which by (\ref{coord9inv}) correspond to square-root type algebraic poles in the variable $y$. 

To show that these are the only types of behaviour that can occur, we have to show that any solution actually traverses either the line $\mathcal{L}_9^+$ or $\mathcal{L}_9^-$. This is achieved by considering the following auxiliary function,
\begin{equation*}
W = \frac{1}{2}(y'(z))^2 - \frac{1}{2} y(z)^6 - \frac{\alpha(z)}{4} y(z)^4 - \frac{\beta(z)}{3} y(z)^3 - \frac{\gamma(z)}{2} y(z)^2 - \delta(z) y(z) + \left( \sum_{k=1}^4 \frac{\xi_k(z)}{y(z)^k} \right) y'(z),
\end{equation*}
where we impose the conditions (\ref{condn=5}) and the functions $\xi_k$ can be determined to be
\begin{equation*}
	\xi_1 = \frac{1}{8} \alpha', \quad \xi_2 = \frac{1}{3} \beta', \quad \xi_3 = 0, \quad \xi_4 = \frac{1}{24} \beta \alpha' + \frac{1}{3} \beta'' - \delta'. 
\end{equation*}
Here, $\xi_3$ turns out to be arbitrary and has been set to $0$.

After each blow-up one can check that, away from the base point, the logarithmic derivative $\frac{W'}{W}$ is bounded, whereas $W$ itself is infinite. Although the expressions for the logarithmic derivative become lengthy and are omitted here, this can be checked easily using a computer algebra system. Lemma \ref{log_bounded} then shows that the lines $\mathcal{L}_0$, $\mathcal{L}_i \setminus \mathcal{P}_{i+1}$, $i=1,2,3$ and $\mathcal{L}^\pm_{i} \setminus \mathcal{P}^\pm_{i+1}$, $i \in \{4,5,6,7,8\}$ are inaccessible for the flow of the vector field. 

Denoting by $\mathcal{S}_9(z)$ the space obtained by blowing up $\mathbb{P}^2$ along the two cascades $\mathcal{P}_1 \leftarrow \cdots \leftarrow \mathcal{P}_9^+$ and $\mathcal{P}_1 \leftarrow \cdots \leftarrow \mathcal{P}_9^-$, and $\mathcal{I}'_8(z)$ the proper transform of the set $\mathcal{I}_8(z) = \mathcal{L}_0 \cup \mathcal{L}_1 \cup \mathcal{L}_2 \cup \mathcal{L}_3 \cup \bigcup_{i=4}^8 \mathcal{L}_i^+ \cup \bigcup_{i=4}^8 \mathcal{L}_i^-$ in $\mathcal{S}_9(z)$, we obtain $\mathcal{S}_9(z) \setminus \mathcal{I}'_8(z)$ as the analogue of the space of initial values for equation (\ref{degree5_reseqn}). 

We can now prove, by similar arguments as in Proposition \ref{degree4prop}, that the algebraic series (\ref{syst9sol}) are the only possible types of movable singularities that can occur by analytic continuation of a solution along finite-length curves.
\begin{prop}
The class of equations
\begin{equation}
\label{degree5_reseqn}
y'' = y^5 + (az+b) y^3 + \beta(z) y^2 + \left( \frac{1}{16}(az+b)^2 + c \right) y + \delta(z),
\end{equation}
where $\beta(z)$ and $\delta(z)$ are analytic in $z$ and $a,b,c \in \mathbb{C}$ constants, has the quasi-Painlev\'e property, with square-root type algebraic poles.
\end{prop}

\begin{proof}
The proof proceeds similar to the proof of Proposition \ref{degree4prop}. Suppose that a solution, analytically continued along a finite-length curve $\gamma \subset \mathbb{C}$, ends in a movable singularity $z_\ast \in \mathbb{C}$. The lifted curve in the extended phase space is denoted $\Gamma(z)$. Let $(z_n) \subset \gamma$, $z_n \to z_\ast$ be a sequence along $\gamma$. Due to the compactness of the phase space (with all exceptional curves), there exists a subsequence $(z_{n_k})$ such that the lifted sequence $\Gamma(z_{n_k})$ converges to a point $P_\ast \in S_9(z_\ast)$. By Lemma \ref{log_bounded}, and the existence of a function $W(z)$ that is infinite on the set $\mathcal{I}'_8(z)$, with bounded logarithmic derivative $\frac{W'}{W}$, we actually have $P_\ast \in S_9(z_\ast) \setminus \mathcal{I}'_8(z_\ast)$. Since the solution has a singularity at $z_\ast$, we must have $P_\ast \in \mathcal{L}^+_9 \cup \mathcal{L}^-_9$. For, if this was not the case, the sequence of points in the phase space would have an accumulation point away from the exceptional curves, where the original equation has a regular initial value problem, and, by Lemma \ref{Painlemma} has an analytic solution, which is contrary to the assumption of a singularity at $z_\ast$.
Now suppose that $P_\ast \in \mathcal{L}^+_9$ (the case for $P_\ast \in \mathcal{L}^-_9$ is similar). The sequence $(z_{n_k}, u^+_9(z_{n_k}),v^+_9(z_{n_k}))$ converges to the point $P_\ast \in \mathcal{L}^+_9$, on which the system (\ref{system9regular}) defines a regular initial value problem for $(z,v^+_9)$ in the variable $u_9^\pm$, resp. $(z,U^+_9)$ in the variable $V^+_9$. Therefore, by Lemma \ref{Painlemma}, system (\ref{system9regular}) has an analytic solution of the form (\ref{system9solution}), with $z_0 = z_\ast$, which translates into a square-root type branch point for $(u^+_9(z),v^+_9(z))$ and therefore by (\ref{coord9inv}) into a square-root type algebraic pole for $y(z)$.
\end{proof}

\section{Hamiltonian systems with algebraic singularities}
\label{sec:HamiltonianSystems}
In the previous section we have seen how to resolve the base points of the second-order equations of the form $y'' = P(z,y)$, extending the phase space of $(y,y')$. These equations are themselves Hamiltonian systems by letting 
\begin{equation}
\label{2nd-order-Hamil}
H(z,x,y) = \frac{1}{2} x^2 - \tilde{P}(z,y), \quad \frac{\partial \tilde{P}}{\partial y} = P(z,y),
\end{equation}
where $x=y'$ and we let $N = \deg_y P$. In the previous sections we considered the cases $N=4$ and $N=5$, whereas the case $N=3$ was discussed in section \ref{sec:InitialValueSpace}, leading to the second Painlev\'e equation. Furthermore, the case $N=2$ leads to the first Painleve equation. In fact, all six Painlev\'e equations can be written as polynomial Hamiltonian systems $H(z,x,y)$ with rational coefficients in $z$. The blow-ups leading to the space of initial values for all Painlev\'e Hamiltonian systems where performed by Okamoto \cite{Okamoto1979}.

In \cite{Kecker2016}, one of the authors studied the class of polynomial Hamiltonian systems,
\begin{equation}
	\label{generalMN_Hamiltonian}
\begin{aligned}
H(z,x(z),y(z)) &= \sum_{i=0}^M \sum_{j=0}^N \alpha_{ij}(z) x(z)^i y(z)^j, \\
x'(z) = \frac{\partial H}{\partial y}, & \quad y'(z) = - \frac{\partial H}{\partial x}, 
\end{aligned}
\end{equation}
with $i,j$ constrained by $iN+jM \leq MN$, so that $x^M$ and $y^N$ are the dominant terms in the equations which, similar to the equations $y'' = P(z,y)$, under certain resonance conditions, have the property that all their movable singularities are algebraic poles. We consider here the case where the coefficients of the dominant terms are constant, which amounts to saying the system (\ref{generalMN_Hamiltonian}) has no fixed singularities. By a suitable scaling, these constants can take any (non-zero) numerical value. Furthermore, the terms $x^{M-1}$ and $y^{N-1}$ can be transformed away, leaving us with the Hamiltonian
\begin{equation}
\label{MNHamiltonian}
H = \frac{1}{N} y^N - \frac{1}{M} x^M + \sum_{0 < iM+jN<MN} \alpha_{ij}(z) x^i y^j.
\end{equation}
Under these assumptions, leading order behaviour for series solutions $(x(z),y(z)$ is of the form
\begin{equation}
\label{HamExpansions}
x(z) = c_0 (z-z_0)^{-\frac{N}{MN-M-N}} + \cdots , \quad y(z) = d_0 (z-z_0)^{-\frac{M}{MN-M-N}} + \cdots .
\end{equation}
Using the method of compactifying the phase space and blowing up the base points, we will see how to obtain the conditions by which the expansions (\ref{HamExpansions}) yield algebraic poles, i.e.\ when they are free from logarithmic singularities.
The case $\min\{M,N\} = 2$ can be reduced essentially to the case (\ref{2nd-order-Hamil}), representing second-order equations. We will thus look at some examples with $M,N \geq 3$. The case $M=N=3$, discussed in the next paragraph, is interesting as it leads to a system of equations related to the fourth Painlev\'e equation, i.e. in this case the singularities are simple (ordinary) poles. The space of initial values for this system was already computed in \cite{Kecker2019} and is reproduced here for completeness. In the following two sections we then consider the cases $M=N=4$ and $M=3,N=4$, which have square-root and $5$th-root type algebraic poles, respectively. Constructing the analogue of the space of initial values for these systems and using an appropriate auxiliary function in conjunction with Lemma \ref{log_bounded} to show that the exceptional curves from the intermediate blow-ups are inaccessible for the flow of the vector field, allows us to conclude that these are the only possible types of movable singularities under analytic continuation along finite-length curves, i.e.\ these systems have the quasi-Painlev\'e property. The forms of the auxiliary functions are taken from the article \cite{Kecker2016}, where they are derived as a quantity that is bounded at all movable singularities.

\subsection{Case $M=N=3$: a system with the Painlev\'e property. \\}
\label{sec:cubicHam}
We consider the cubic Hamiltonian system
\begin{equation}
\label{cubicHam}
H(z,x(z),y(z)) = \frac{1}{3} \left( y^3 - x^3 \right) + \gamma(z)xy +\beta(z) x + \alpha(z) y,
\end{equation}
which was introduced in \cite{Kecker2015}. If $\alpha$ and $\beta$ are constants and $\gamma(z)$ a function at most linear in $z$, it was shown that the system of equations derived from (\ref{cubicHam}) has the Painlev\'e property. Below we will see that, by applying the procedure of compactifying the system (\ref{cubicHam}) with general analytic functions $\alpha(z)$, $\beta(z)$, $\gamma(z)$, after blowing up and resolving the base points of the system, these conditions can be recovered by requiring that the system has no logarithmic singularities.

Extending the system to projective space we obtain, in the three standard coordinate charts of $\mathbb{P}^2$, $[1:y:x] = [u:v:1] = [V:1:U]$,
\begin{equation*}
\begin{aligned}
x'(x) &= y^2 + \gamma(z) x + \alpha(z), & \quad
y'(z) &= x^2 - \gamma(z) y - \beta(z), \\
u'(z) &=-v^2-u^2 \alpha (z)-u \gamma (z), & \quad
v'(z) &=-\frac{-1+v^3+u^2 v \alpha (z)+u^2 \beta (z)+2 u v \gamma(z)}{u},\\
V'(z) &=-U^2+V^2 \beta (z)+V \gamma(z), & \quad U'(z) &=-\frac{-1+U^3-V^2 \alpha (z)-U V^2 \beta (z)-2 U V \gamma(z)}{V}.
\end{aligned}
\end{equation*}
We can see that initially there are three base points on the line at infinity of $\mathbb{P}^2$, given by
\begin{equation*}
\mathcal{P}_1^\rho: (u,v) = (0,\rho) \quad \leftrightarrow \quad (U,V) = (\rho^{-1},0), \quad \rho \in \{1,\omega,\bar{\omega}\},
\end{equation*}
where $\omega = \frac{-1 + i \sqrt{3}}{2}$ is a third root of unity and $\bar{\omega}$ its complex conjugate. Keeping $\rho$ as a symbol representing either of the three roots of unity, each base point is resolved by a cascade of three blow-ups. We denote the coordinates of the three respective sequences of blow-ups with superscripts $\rho \in  \{1,\omega,\bar{\omega}\}$:
\begin{equation*}
\begin{aligned}
{} & \mathcal{P}_1^\rho: (u^\rho,v^\rho) = \left(\frac{1}{x},\frac{y}{x}\right) = (0,\rho) \quad \leftarrow \quad \mathcal{P}_2^\rho: (u_1^\rho,v_1^\rho) = \left( \frac{1}{x}, y - \rho x \right) = (0,-\bar{\rho} \gamma(z)) \\
& \leftarrow \quad \mathcal{P}_3^\rho: (u_2^\rho,v_2^\rho) = \left( \frac{1}{x}, x \left( y - \rho x + \bar{\rho} \gamma(z) \right) \right) = \left( 0, \gamma'(z) - \rho \beta(z) - \bar{\rho} \alpha(z) \right).
\end{aligned}
\end{equation*}
After blowing up $\mathcal{P}_3^\rho$, the system of equations takes the following form:
\begin{equation}
\label{cubicsystem3}
\begin{aligned}
u_3^\rho{}' &= p^{\rho}_{3,1}(z,u_3^\rho,v_3^\rho), \\
v_3^\rho{}' &= \frac{\bar{\rho} \alpha'(z) + \rho \beta'(z) - \gamma''(z) + p_{3,2}^\rho(z,u_3^\rho,v_3^\rho)}{u_3^\rho}, \\
U_3^\rho{}' &= \frac{U_3^\rho (\gamma''(z) - \bar{\rho} \alpha'(z) - \rho \beta'(z)) + P_{3,1}(z,U_3^\rho,V_3^\rho)}{V_3^\rho},\\
V_3^\rho{}' &= \frac{-\bar{\rho} + P_{3,2}(z,U_3^\rho,V_3^\rho)}{U_3^\rho}.
\end{aligned}
\end{equation}
We see that there is an additional base point at $\mathcal{P}^\rho_4: U_3^\rho = V_3^\rho = 0$. This point can be blown up once more, rendering a system free from base points (but not regular). However, the point $\mathcal{P}_4^\rho$ is only present if the condition 
\begin{equation}
\label{cubicHam_cond}
\bar{\rho} \alpha'(z) + \rho \beta'(z) - \gamma''(z) \equiv 0,
\end{equation}
is {\it not} satisfied, in which case the system (\ref{cubicsystem3}) exhibits solutions with logarithmic singularities. If, however, condition (\ref{cubicHam_cond}) is satisfied, factors of $u_3^\rho$ and $V_3^\rho$ cancel in the second resp. third equation of system (\ref{cubicsystem3}), which then defines a regular initial value problem at every point on the exceptional curves $\mathcal{L}_3^\rho: (u_3^\rho,v_3^\rho) = (0,h)$ of the third blow-up in each cascade. Denoting by $\mathcal{S}_3(z)$ the space obtained by blowing up $\mathbb{P}^2$ along the three cascades of base points $\mathcal{P}_1^\rho \leftarrow \mathcal{P}_2^\rho \leftarrow \mathcal{P}_3^\rho$, $\rho \in \{1,\omega,\bar{\omega}\}$, the space of initial values is $\mathcal{S}_3(z) \setminus \left( \bigcup_{\rho \in \{1,\omega,\bar{\omega}\}} \mathcal{I}^{\rho}_2{}'(z) \right)$, where $\mathcal{I}^\rho_2{}'(z)$ is the union of the proper transforms of the line at infinity $\mathcal{L}_0 = I \setminus \{ \mathcal{P}^1_1, \mathcal{P}^\omega_1, \mathcal{P}^{\bar{\omega}}_1 \} \subset \mathbb{P}^2$ and the exceptional curves $\mathcal{L}_i^\rho$, $i=1,2$, from the first two blow-ups in each cascade of base points. 

Together, the three conditions (\ref{cubicHam_cond}), for $\rho \in \{1,\omega,\bar{\omega}\}$, are required for the absence of logarithms in the solutions, which result in $\gamma'' = \beta' = \alpha' \equiv 0$, that is, $\alpha$ and $\beta$ are constant and $\gamma(z) = az+b$ is at most linear in $z$. In case $a=0$, the Hamiltonian system is autonomous and can be integrated directly using the Hamiltonian as first integral. When $a \neq 0$, by a re-scaling of $z$, $x$ and $y$ the system can be normalised to the form
\begin{equation}
\label{cubicPainlevesystem}
H = \frac{1}{3}(y^3 - x^3) + zxy + \alpha y + \beta x, \quad x' = y^2 + zx + \alpha, \quad y' = x^2 - zy - \beta.
\end{equation}
This system is in fact closely related to the Hamiltonian system defining the fourth Painlev\'e equation, and was introduced in \cite{Kecker2015} and investigated further in \cite{Steinmetz2018}. By similar arguments as in section \ref{sec:InitialValueSpace}, constructing the space of initial values gives an alternative method of proof for the Painlev\'e property of system (\ref{cubicPainlevesystem}). The proof makes use of the following auxiliary function,
\begin{equation}
	\label{MN3_W}
	W = H - \frac{y^2}{x}.
\end{equation}
The correction term $\frac{y^2}{x}$ is chosen to compensate for the divergence of $H' = \frac{\partial H}{\partial z} = x y$ at any singularity (one could alternatively have chosen $\frac{x^2}{y}$ due to the symmetry in $x$ and $y$). In \cite{Kecker2015} it is shown that $W$ satisfies a first-order differential equation of the form $W' = P W + Q$, where $P$ and $Q$ are bounded functions, and hence that $W$ itself is bounded. In the context of the space of initial values, we need to check that the logarithmic derivative $\frac{W'}{W}$ is bounded on the exceptional curve, whereas $W$ itself is infinite there, to apply Lemma \ref{log_bounded}. We demonstrate the process for this example, re-writing the function $W$ in terms of the other coordinate charts of $\mathbb{P}^2$.

In the chart $[u:v:1]$, we have 
\begin{equation}
	W_0(z,u(z),v(z)) = (v^3-1-3 u^2 v^2+3 u v z+3 u^2 v \alpha +3 u^2 \beta)/(3 u^3).
\end{equation}
The logarithmic derivative is 
\begin{equation}
	\frac{d \log W_0}{d z} = \frac{W_0'}{W_0} = \frac{3 u \left(-v+v^4+3 u v^2 z+u^2 v^2 \alpha +2 u^2 v \beta \right)}{v^3-1-3 u^2 v^2+3 u v z+3 u^2 v \alpha +3 u^2 \beta },
\end{equation}
which is bounded in a neighbourhood of any point on the line $u=0$, apart from the base points where $v^3=1$. Therefore, by Lemma \ref{log_bounded}, the line at infinity is inaccessible for the solution away from the points $(u,v)=(0,\rho)$, $\rho \in  \{1,\omega,\bar{\omega}\}$. After the first blow-up, in the coordinates $(u_1,v_1) = (u^\rho_1,v^\rho_1)$, where for simplicity we let $\rho =1$, we have
\begin{equation*}
	W_1(z,u_1,v_1) = \frac{-1+3 \beta  u_1^2+ (3 z u_1 + 3 \alpha  u_1^2) \left(1+u_1 v_1\right) -3 u_1^2 \left(1+u_1 v_1\right)^2+\left(1+u_1 v_1\right)^3}{3u_1^3},
\end{equation*}
and
\begin{equation*}
	\frac{W_1'}{W_1} = \frac{3 u_1  P_1(u_1,v_1,z)}{3 v_1+3 z +u_1 Q_1(u_1,v_1,z)},
\end{equation*}
where
\begin{equation*}
	\begin{aligned}
	P_1 &= 3 z+\alpha  u_1+2 \beta  u_1+3 v_1+6 z u_1 v_1+2 \alpha  u_1^2 v_1+2 \beta  u_1^2 v_1+6 u_1 v_1^2+3 z u_1^2 v_1^2+\alpha  u_1^3 v_1^2+4 u_1^2 v_1^3+u_1^3 v_1^4, \\
	Q_1 &= -3 +3 \alpha +3 \beta +3 z v_1-6 u_1 v_1+3 \alpha  u_1 v_1+3 v_1^2-3 u_1^2 v_1^2+u_1 v_1^3,
	\end{aligned}
\end{equation*}
the calculations for $\rho = \omega,\bar{\omega}$ being similar. $\frac{W_1'}{W_1}$ is bounded in a neighbourhood of any point on the exceptional line $u_1=0$, other than the base point $(u_1,v_1)=(0,-z)$. Supposing that we are analytically continuing a solution leading up to a singularity at $z_\ast$, the point $(u_1,v_1)=(0,-z_\ast)$ is the only point on the exceptional curve where Lemma \ref{log_bounded} cannot be applied. 
Performing the second blow-up we have, in the coordinates $(u_2,v_2)$, 
\begin{equation*}
	W_2(z,u_2,v_2) = \frac{-1-3 u^2 (1+u (u v-z))^2+(1+u (u v-z))^3+ (3 u +3 u^2 \alpha) (1+u (u v-z)) z +3 u^2 \beta }{3 u^3},
\end{equation*}
and
\begin{equation*}
	\frac{W_2'}{W_2} = \frac{3 u P_2(u,v,z)}{3 v - 3 + 3 \alpha + 3 \beta + u Q_2(u,v,z)},
\end{equation*}
where
\begin{equation*}
	\begin{aligned}
	P_2 & = \left(3 v+6 u^2 v^2+4 u^4 v^3+u^6 v^4-6 u v z-9 u^3 v^2 z-4 u^5 v^3 z+6 u^2 v z^2+6 u^4 v^2 z^2-u z^3 \right. \\ & \left. -4 u^3 v z^3+u^2 z^4+\alpha +2 u^2 v \alpha +u^4 v^2 \alpha -2 u z \alpha -2 u^3 v z \alpha +u^2 z^2 \alpha +2 \beta +2 u^2 v \beta -2 u z \beta \right), \\
	Q_2 & =  -6 u v+3 u v^2-3 u^3 v^2+u^3 v^3+6 z-3 v z+6 u^2 v z-3 u^2 v^2 z-3 u z^2+3 u v z^2- z^3 +3 u v \alpha -3 z \alpha,
	\end{aligned}
\end{equation*}
so that $\frac{W_2'}{W_2}$ is bounded in a neighbourhood of any point on the exceptional curve $u_2=0$, other than $(u_2,v_2) = (0, 1-\alpha -\beta)$, whereas $W_2$ itself is infinite on this line. By Lemma \ref{log_bounded} we can conclude that, if a solution approaches a singularity $z_\ast$ along some finite-length path $\gamma$ ending in $z_\ast$, it must pass through one of the exceptional lines $\mathcal{L}^\rho_3$ introduced by the third blow-ups in each cascade. On the lines $\mathcal{L}^\rho_3$ the function $W$, re-written in the appropriate coordinates, is in fact finite. Furthermore, any solution approaching the lines $\mathcal{L}^{\rho}_3$ can be analytically continued across these lines where the system defines a regular initial value problem. The solutions on the lines $\mathcal{L}^{\rho}_3$, when transformed back into the original coordinates, result in simple poles for $x(z),y(z)$ with residues $-\rho$ and $\bar{\rho}$, respectively. This provides an alternative proof that the Hamiltonian system defined by (\ref{cubicPainlevesystem}) has the Painlev\'e property.

\subsection{Case $M=N=4$.}
The differential system studied in this section arises from the Hamiltonian
\begin{equation*}
	H(z,x,y) = \frac{1}{4} \left(y^4 - x^4 \right) + \sum_{0<i+j\leq 3} \alpha_{i,j}(z) x^i y^j,
\end{equation*}
where the $\alpha_{i,j}(z)$ are analytic functions in some common domain $\Omega \subset \mathbb{C}$. Similar as in the case $M=N=3$, the geometry of the analogue of the space of initial values of this system is much more symmetric than in the case of the second-order equations discussed in sections \ref{sec:degree4} and \ref{sec:degree5}. Although $16$ blow-ups are required to regularise the system, these decompose into $4$ separate cascades of $4$ blow-ups. As before, we write down the extended system of equations in the three standard charts of $\mathbb{P}^2$:
\begin{equation*}
	\begin{aligned}
		x' &= y^3+\alpha _{0,1}+2 y \alpha _{0,2}+x \alpha _{1,1}+2 x y \alpha _{1,2}+x^2 \alpha_{2,1}, \\
		y' &= x^3-\alpha _{1,0}-y \alpha _{1,1}-y^2 \alpha _{1,2}-2 x \alpha _{2,0}-2 x y \alpha_{2,1}, \\
		u' &= -\frac{v^3+u^3 \alpha _{0,1}+2 u^2 v \alpha _{0,2}+u^2 \alpha _{1,1}+2 u v \alpha _{1,2}+u \alpha_{2,1}}{u}, \\
		v' &= -\frac{-1+v^4+u^3 v \alpha _{0,1}+2 u^2 v^2 \alpha _{0,2}+u^3 \alpha _{1,0}+2 u^2 v \alpha _{1,1}+3 u v^2 \alpha _{1,2}+2 u^2 \alpha _{2,0}+3 u v \alpha _{2,1}}{u^2}, \\
		U' &= -\frac{-1+U^4-V^3 \alpha _{0,1}-2 V^2 \alpha _{0,2}-U V^3 \alpha _{1,0}-2 U V^2 \alpha _{1,1}-3 U V \alpha _{1,2}-2 U^2 V^2 \alpha _{2,0}-3 U^2 V \alpha _{2,1}}{V^2}, \\
		V' &= -\frac{U^3-V^3 \alpha_{1,0}-V^2 \alpha _{1,1}-V \alpha _{1,2}-2 U V^2 \alpha _{2,0}-2 U V \alpha _{2,1}}{V}.
	\end{aligned}
\end{equation*}
We observe that there are six base points (denoted with superscripts), namely
\begin{equation*}
	\begin{aligned}
		\mathcal{P}^0_1 &: (u,v) = (0,0) \\
		\tilde{\mathcal{P}}^0_1 &: (U,V) = (0,0) \\
		\mathcal{P}_1^1 &: (u,v) = (0,1) & \leftrightarrow & \quad (U,V) = (1,0) \\
		\mathcal{P}_1^i &: (u,v) = (0,i) & \leftrightarrow & \quad (U,V) = (-i,0) \\
		\mathcal{P}_1^{-1} &: (u,v) = (0,-1) & \leftrightarrow & \quad (U,V) = (-1,0) \\
		\mathcal{P}_1^{-i} &: (u,v) = (0,-i) & \leftrightarrow & \quad (U,V) = (i,0). \\
	\end{aligned}
\end{equation*}
We note that the points $\mathcal{P}^0_1$ and $\tilde{\mathcal{P}}^1_0$ can be resolved by one blow-up each, only the transforms of the points $\mathcal{P}_1^\rho$, $\rho \in \{1,i,-1,-i\}$ are still visible in the charts obtained after blowing up $\mathcal{P}^0_1$ and $\tilde{\mathcal{P}}^0_1$.

We now resolve the four base points $\mathcal{P}_1^\rho$, $\rho \in \{1,i,-1,-i\}$, in the coordinates $(u,v)$, where we use the superscript $\rho$ to denote the coordinates after blowing up. For each point $\mathcal{P}_1^\rho$, we find the following cascade of four blow-ups:
\begin{equation*}
	\begin{aligned}
		{} & \mathcal{P}_1^\rho: (u,v) = \left(\frac{1}{x},\frac{y}{x}\right) = (0,\rho) \quad \leftarrow \quad \mathcal{P}_2^\rho: (u_1^\rho,v_1^\rho) = \left( \frac{1}{x}, y - \rho x \right) = (0,\alpha_{2,1} + \rho \alpha_{1,2}) \\
		& \leftarrow \quad \mathcal{P}_3^\rho: (u_2^\rho,v_2^\rho) = \left( \frac{1}{x}, x \left(y - \rho x + \bar{\rho} \alpha_{1,2} + \rho^2 \alpha_{2,1} \right) \right) = \left( 0, \frac{\rho}{2} \alpha_{1,2}^2 - \frac{\bar{\rho}}{2} \alpha_{2,1} -\rho^2 \alpha_{1,1} - \bar{\rho} \alpha_{0,2} - \rho \alpha_{2,0} \right) \\
		& \leftarrow \quad \mathcal{P}_4^\rho: (u_3^\rho,v_3^\rho) = \left( \frac{1}{x}, x \left(- \rho x^2 + x y + \bar{\rho} \alpha _{0,2} + \rho^2 \alpha_{1,1} + \bar{\rho} x \alpha _{1,2} - \frac{\rho}{2} \alpha_{1,2}^2 + \rho \alpha_{2,0} + \rho^2 \alpha _{2,1} + \frac{\bar{\rho}}{2} \alpha _{2,1}^2\right) \right) \\
		& \qquad \qquad = \left(0, -\frac{i}{2} \left(2 i \alpha _{0,1}+2 \alpha _{1,0}-2 \alpha _{0,2} \alpha _{1,2}+\alpha _{1,2}^3-2
		\alpha _{1,2} \alpha _{2,0}-2 \alpha _{1,1} \alpha _{2,1} -4 i \alpha _{1,2}^2 \alpha_{2,1}  +4 i \alpha _{2,0} \alpha_{2,1} \right. \right. \\ & \qquad \qquad \qquad \left. \left.  -3 \alpha _{1,2} \alpha _{2,1}^2+2 i \alpha _{1,2}'+2 \alpha_{2,1}'\right) \right).
	\end{aligned}
\end{equation*}

After blowing up $\mathcal{P}_4^\rho$, the system of equations takes the following form:
\begin{equation}
	\label{system4}
	\begin{aligned}
		u_4^\rho{}' & = \frac{- \bar{\rho} + p_{4,1}(z,u_4^\rho,v_4^\rho)}{u_4^\rho}, \\
		v_4^\rho{}' & = \frac{\rho^2 \alpha_{1,1}'(z) + \rho \left( \alpha_{2,0}'(z) - \alpha _{1,2}(z) \alpha _{1,2}'(z) \right) + \bar{\rho} \left( \alpha _{0,2}'(z) + \alpha _{2,1}(z) \alpha _{2,1}'(z) \right) + p_{4,2}(z,u_4^\rho,v_4^\rho) }{\left( u_4^\rho \right)^2}, \\
		U_4^\rho{}' &= \frac{-\rho^2 \alpha_{1,1}'(z) - \rho \left( \alpha_{2,0}'(z) - \alpha _{1,2}(z) \alpha _{1,2}'(z) \right) - \bar{\rho} \left( \alpha _{0,2}'(z) + \alpha _{2,1}(z) \alpha _{2,1}'(z) \right) + P_{4,1}(z,U_4^\rho,V_4^\rho)}{(V_4^\rho)^2}, \\
		V_4^\rho{}' &= \frac{- \bar{\rho} + P_{4,2}(z,U_4^\rho,V_4^\rho)}{(U_4^\rho)^2 V_4^\rho}.
	\end{aligned}
\end{equation}
Thus, unless the condition
\begin{equation}
	\label{mn4cond}
	\rho^2 \alpha_{1,1}'(z) + \rho \left( \alpha_{2,0}'(z) - \alpha _{1,2}(z) \alpha _{1,2}'(z) \right) + \bar{\rho} \left( \alpha _{0,2}'(z) + \alpha _{2,1}(z) \alpha _{2,1}'(z) \right) = 0
\end{equation}
is satisfied, the system admits logarithmic singularities,
\begin{equation*}
	\begin{aligned}
		u_4^\rho =& (-2\bar{\rho})^{1/2} (z-z_0)^{1/2} + O(z-z_0), \\
		v_4^\rho =& \left( \rho^2 \alpha_{1,1}'(z_0) + \rho \left( \alpha_{2,0}'(z_0) - \alpha _{1,2}(z_0) \alpha _{1,2}'(z_0) \right) + \bar{\rho} \left( \alpha _{0,2}'(z_0) + \alpha _{2,1}(z_0) \alpha _{2,1}'(z_0) \right) \right) \log(z-z_0) \\ & + O((z-z_0)^{1/2}).
	\end{aligned}
\end{equation*}
For the solutions of the system to be free from logarithmic branch points, condition (\ref{mn4cond}) must be satisfied for all $\rho \in \{1,i,-1,-i\}$. Then, one factor of $u_4^\rho$ and $V_4^\rho$ cancel in the second resp.\ third equation of system (\ref{system4}).

To see that the singularities $z_\ast$ of the solution are all of this form, we write the system in the form
\begin{equation*}
	\begin{aligned}
		\frac{d z}{d u_4^\rho} &= \frac{u_4^\rho}{- \bar{\rho} + p_{4,1}(z,u_4^\rho,v_4^\rho)}, \\
		\frac{d v_4^\rho}{d u_4^\rho} &= \frac{p_{4,2}(z,u_4^\rho,v_4^\rho)}{- \bar{\rho} + p_{4,1}(z,u_4^\rho,v_4^\rho)},
	\end{aligned}
\end{equation*}
where we have interchanged the role of the dependent and independent variables. This system has analytic solutions for initial values $(z,u_4^\rho,v_4^\rho) = (z_\ast,0,h)$ on the exceptional curve $\mathcal{L}_4^\rho$,
\begin{equation*}
	z = z_\ast - \rho (u_4^\rho)^2 + O((u_4^\rho)^3), \quad v_4^\rho = h + O(u_4^\rho),
\end{equation*}
which can be inverted to find square-root type algebraic series expansions for $(u_4^\rho,v_4^\rho)$:
\begin{equation*}
	u_4^\rho = (z-z_\ast)^{1/2} + O(z-z_\ast), \quad v_4^\rho = h + (z-z_\ast)^{1/2} + O(z-z_\ast).
\end{equation*}
The conditions (\ref{mn4cond}), for $\rho \in \{1,i,-1,-i\}$, decouple into three linearly independent conditions among the $\alpha_{i,j}(z)$ and their derivatives, namely
\begin{equation*}
	(2 \alpha_{2,0}(z) - \alpha_{1,2}(z)^2)' = \alpha_{1,1}'(z) = (2 \alpha_{0,2}(z) + \alpha_{2,1}(z)^2)' \equiv 0,
\end{equation*}
that is, the functions $2 \alpha_{2,0}(z) - \alpha_{1,2}(z)^2$, $\alpha_{1,1}(z)$ and $2 \alpha_{0,2}(z) + \alpha_{2,1}(z)^2$ each have to be equal to a constant. This is in agreement with the resonance conditions found in \cite{Kecker2016} for this Hamiltonian system. Furthermore, we introduce the following auxiliary function,
\begin{equation}
	W = H - \alpha_{2,1}'(z) \frac{y^2}{x} - \alpha_{1,2}'(z) \frac{y^3}{x^2}.
\end{equation}
Using computer algebra, after each blow-up one can routinely check that the logarithmic derivative of $W$ is bounded in a neighbourhood of any point on the exceptional curves away from the base points, while $W$ is infinite. Lemma \ref{log_bounded} then guarantees that the exceptional curves introduced by the first three blow-ups of each cascade are inaccessible for the flow of the vector field. Let $\mathcal{S}_4(z)$ denote the space obtained by blowing up $\mathbb{P}^2$ along the four cascades of base points, $\mathcal{P}_1^\rho \leftarrow \mathcal{P}_2^\rho \leftarrow \mathcal{P}_3^\rho \leftarrow \mathcal{P}_4^\rho$, $\rho \in \{1,i,-1,-i\}$, and $\mathcal{I}_3(z) = \mathcal{L}_0 \cup \bigcup_{i=1}^3 \mathcal{L}_i^1 \cup \bigcup_{i=1}^3 \mathcal{L}_i^i \cup \bigcup_{i=1}^3 \mathcal{L}_i^{-1} \cup \bigcup_{i=1}^3 \mathcal{L}_i^{-i}$. The space of initial values for the Hamiltonian system (\ref{MN4_resHam}) is $\mathcal{S}_4(z) \setminus \mathcal{I}_3'(z)$, at each point of which the system either defines an analytic solution or a solution with square-root type algebraic branch point. Using similar arguments as in the previous sections we can thus show:

\begin{prop}
	\label{MN4_prop}
	Given the Hamiltonian
	\begin{equation}
		\label{MN4_resHam}
		H = \frac{1}{4}\left( y^4 - x^4 \right) + \alpha_{2,1} x^2 y + \alpha_{1,2} x y^2 + (a + \frac{1}{2} \alpha_{2,1}^2) x^2 + (b - \frac{1}{2} \alpha_{1,2}^2)y^2 + c x y + \alpha_{1,0} x + \alpha_{0,1} y,
	\end{equation}
	where $\alpha_{1,2}(z),\alpha_{2,1}(z),\alpha_{1,0}(z),\alpha_{0,1}(z)$ are analytic functions and $a,b,c \in \mathbb{C}$ are constants, the system derived from this Hamiltonian has the quasi-Painlev\'e property, with square-root type algebraic poles.
\end{prop}

\subsection{Case $M=3,N=4$.}
With a slightly different normalisation as given in (\ref{MNHamiltonian}) we consider the class of Hamiltonians
\begin{equation}
\label{M3N4Hamiltonian}
H(z,x(z),y(z)) = y^4 - x^3 + \sum_{0<i+j\leq 3} \alpha_{ij}(z) x(z)^i y(z)^j.
\end{equation}
This only differs from the preceding case by the power of $x$ being one less. However since the system is no longer symmetric in $x$ and $y$, the blow-up structure in this case is very different. We will see that a single cascade of $16$ blow-ups is necessary to resolve an initial base point. In fact, this example is more similar to the second-order equation in section \ref{sec:degree4}. Extending the Hamiltonian system derived from (\ref{M3N4Hamiltonian}) to $\mathbb{P}^2$ yields the three systems of equations
\begin{align*}
x'(z) &= 4 y^3 + a_{21} x^2 + 2 a_{12} x y + a_{11} x + 2 a_{02} y + a_{01}, \\
y'(z) &= 3 x^2 - 2 a_{21} x y - a_{12} y^2 - 2 a_{20} x - a_{11} y - a_{10}, \\
u'(z) &= -\frac{2 u^2 v \alpha_{02}+u^3 \alpha_{01}+u^2 \alpha _{11}+2 u v \alpha _{12}+u \alpha_{21}+4 v^3}{u}, \\
v'(z) &= -\frac{2 u^2 v^2 \alpha_{02} +u^3 v \alpha_{01}+2 u^2 v \alpha_{11}+u^3 \alpha_{10}+2 u^2 \alpha _{20}+3 u v^2 \alpha_{12}+3 u v \alpha_{21}-3 u+4 v^4}{u^2}, \\
U'(z) &= -\frac{-2 U^2 V^2 \alpha_{20}-3 U^2 V \alpha _{21}+3 U^3 V-U V^3 \alpha_{10}-2 U V^2 \alpha_{11}-3 U V \alpha _{12}-V^3 \alpha_{01}-2 V^2 \alpha_{02}-4}{V^2}, \\
V'(z) &= -3 U^2+2 U V \alpha_{20}+2 U\alpha_{21}+V^2 \alpha_{10}+V \alpha_{11}+\alpha_{12}, 
\end{align*}
which have a single base point at $(u,v) = (0,0)$. This indeterminacy of the vector field can be removed by a cascade of $16$ blow-ups, obtained using computer algebra, which we give in the following. Since the expressions for the coordinate transformation become very long in this case, we only give the locations of base points to be blown up:

\begin{align*}
\mathcal{P}_1:& \quad (u,v) = (0,0) \quad \leftarrow \quad \mathcal{P}_2: \quad (U_1,V_1) = (0,0) \quad \leftarrow \quad \mathcal{P}_3: \quad (U_2,V_2) = (0,0) \\
\leftarrow \quad \mathcal{P}_4:& \quad (U_3,V_3) = (0,0) \quad \leftarrow \quad \mathcal{P}_5: \quad (U_4,V_4) = (1,0) \quad \leftarrow \quad \mathcal{P}_6: \quad (U_5,V_5) = (\alpha_{21},0)\\
\leftarrow \quad \mathcal{P}_7:& \quad (U_6,V_6) = \left(\alpha_{12} + \alpha_{21}^2, 0 \right) \quad  \leftarrow \quad \mathcal{P}_8: \quad (U_7,V_7) = \left( 2 \alpha_{12} \alpha_{21} + \alpha_{21}^3, 0 \right) \\
\leftarrow \quad \mathcal{P}_9: & \quad  (U_8,V_8) = \left( \alpha _{21}^4+3 \alpha _{12} \alpha _{21}^2+\alpha _{12}^2+\alpha _{20}, 0 \right) \\
\leftarrow \quad \mathcal{P}_{10}:& \quad (U_9,V_9) = \left( \alpha _{21}^5+4 \alpha_{12} \alpha_{21}^3+3 \alpha_{12}^2 \alpha_{21}+3 \alpha _{20} \alpha_{21}+\alpha_{11}, 0 \right) \\
\leftarrow \quad \mathcal{P}_{11}:& \quad (U_{10},V_{10}) = \left( -\frac{1}{6} \alpha _{12}'+\alpha _{21}^6 + 5 \alpha_{21}^4 \alpha_{12} + 6 \alpha_{12}^2 \alpha_{21}^2 +\alpha_{12}^3 + 6 \alpha_{20} \alpha_{21}^2 +3 \alpha _{11} \alpha _{21} \right. \\ & \qquad \qquad \left. + 3\alpha_{12} \alpha_{20} +\alpha_{02}, 0 \right) \\
\leftarrow \quad \mathcal{P}_{12}:& \quad  (U_{11},V_{11}) = \left( -\frac{19}{30} \alpha _{21} \alpha_{12}'-\frac{1}{5} \alpha _{12}'+\alpha _{21}^7+6 \alpha _{12} \alpha _{21}^5+10 \alpha_{12}^2 \alpha _{21}^3 +10 \alpha_{20} \alpha_{21}^3 \right. \\ & \qquad \qquad \left. +4 \alpha_{12}^3 \alpha _{21} +3 \alpha_{02} \alpha _{21}+12 \alpha_{12} \alpha_{20} \alpha_{21} +3 \alpha _{11} + 6 \alpha_{11} \alpha_{21}^2 + 3 \alpha_{11} \alpha_{12} , 0 \right) \\
\leftarrow \quad \mathcal{P}_{13}:& \quad (U_{12},V_{12}) = \left( -\frac{3}{4} \alpha _{21} \alpha_{12}' -\frac{3}{2} \alpha _{21}^2 \alpha _{12}'-\frac{7}{12} \alpha_{12} \alpha _{12}'+\alpha _{21}^8 +7 \alpha _{12} \alpha _{21}^6+15 \alpha _{12}^2 \alpha_{21}^4 \right. \\ & \qquad \qquad +15 \alpha _{20} \alpha _{21}^4+10 \alpha _{11} \alpha _{21}^3 +10 \alpha_{12}^3 \alpha _{21}^2+30 \alpha _{12} \alpha_{20} \alpha_{21}^2+12 \alpha _{11} \alpha _{12} \alpha_{21} + \alpha_{12}^4 \\ & \qquad \qquad \left. +2 \alpha _{20}^2+\alpha _{10}+6 \alpha _{12}^2 \alpha_{20}+3 \alpha_{02} \left( 2 \alpha_{21}^2+\alpha _{12}\right), 0 \right) \\ 
\leftarrow \quad \mathcal{P}_{14}:& \quad (U_{13},V_{13}) = \left( -\frac{17}{6} \alpha _{21}^3 \alpha _{12}'-\frac{7}{4} \alpha _{21}^2 \alpha _{12}'-\frac{11}{4} \alpha_{12} \alpha _{21} \alpha _{12}'-\frac{2}{3} \alpha _{12} \alpha_{12}' -\frac{1}{3} \alpha_{20}' \right. \\ & \qquad \qquad +\alpha _{21}^9+8 \alpha _{12} \alpha _{21}^7+21 \alpha _{12}^2 \alpha _{21}^5 + 21\alpha _{20} \alpha _{21}^5 +20 \alpha _{12}^3 \alpha _{21}^3+10 \alpha _2 \alpha _{21}^3 \\ & \qquad \qquad +60 \alpha_{12} \alpha_{20} \alpha _{21}^3+5 \alpha _{12}^4 \alpha _{21} +10 \alpha _{20}^2 \alpha_{21}+4 \alpha_{10} \alpha _{21}+12 \alpha _2 \alpha _{12} \alpha _{21} \\ & \qquad \qquad \left. +30 \alpha_{12}^2 \alpha _{20} \alpha _{21} +\alpha _{11} \left(15 \alpha _{21}^4+30 \alpha_{12} \alpha _{21}^2+6 \alpha _{12}^2+4 \alpha _{20}\right) +\alpha_{01}, 0\right) \\ 
\leftarrow \quad \mathcal{P}_{15}:& \quad (U_{14},V_{14}) = \left( -\frac{14}{3} \alpha _{21}^4 \alpha _{12}'-\frac{13}{4} \alpha _{21}^3 \alpha _{12}'-\frac{31}{4} \alpha_{12} \alpha _{21}^2 \alpha _{12}' -\frac{37}{12} \alpha _{12} \alpha _{21} \alpha
_{12}' \right. \\ & \qquad \qquad -\frac{5}{3} \alpha _{21} \alpha _{20}' -\frac{1}{2} \alpha _{11}'-\frac{5}{4} \alpha_{12}^2 \alpha _{12}' -\alpha_{20} \alpha_{12}'+\alpha _{21}^{10}+9 \alpha _{12} \alpha _{21}^8+28 \alpha _{12}^2 \alpha _{21}^6 \\ & \qquad \qquad +28 \alpha _{20} \alpha _{21}^6 +35 \alpha _{12}^3 \alpha_{21}^4+15 \alpha _2 \alpha _{21}^4+105 \alpha _{12} \alpha _{20} \alpha _{21}^4+15 \alpha_{12}^4 \alpha _{21}^2 \\ & \qquad \qquad +30 \alpha _{20}^2 \alpha _{21}^2 +30 \alpha _2 \alpha _{12} \alpha_{21}^2+90 \alpha _{12}^2 \alpha _{20} \alpha _{21}^2+4 \alpha_{01} \alpha _{21} +\alpha _{12}^5 \\ & \qquad \qquad +\alpha_{11} \left(21 \alpha _{21}^4+60 \alpha_{12} \alpha _{21}^2+30 \alpha _{12}^2+20 \alpha_{20}\right) \alpha _{21} +2 \alpha_{11}^2+6 \alpha_{02} \alpha _{12}^2 \\& \qquad \qquad \left. +10 \alpha _{12} \alpha _{20}^2+10 \alpha _{12}^3 \alpha _{20}+4 \alpha _2 \alpha _{20} +2 \alpha_{10} \left(5 \alpha _{21}^2+2 \alpha _{12}\right), 0 \right) \displaybreak \\
\leftarrow \quad \mathcal{P}_{16}:& \quad (U_{15},V_{15}) = \left( -7 \alpha _{21}^5 \alpha _{12}'-\frac{21}{4} \alpha_{21}^4 \alpha _{12}'-\frac{203}{12} \alpha_{12} \alpha _{21}^3 \alpha_{12}'  -\frac{17}{2} \alpha _{12} \alpha _{21}^2 \alpha _{12}'-5 \alpha_{21}^2 \alpha _{20}' \right. \\ & \qquad \qquad-\frac{5}{2} \alpha_{21} \alpha _{11}'-7 \alpha _{12}^2 \alpha_{21} \alpha_{12}' -\frac{29}{5} \alpha _{20} \alpha _{21} \alpha_{12}'-\alpha _2'-\frac{4}{3} \alpha_{12}^2 \alpha _{12}'-\alpha _{11} \alpha_{12}'-\frac{6}{5} \alpha _{20} \alpha_{12}' \\ & \qquad \qquad -\frac{5}{3} \alpha _{12} \alpha_{20}'+\frac{1}{6} \alpha _{21}''+\alpha _{21}^{11}+10 \alpha_{12} \alpha _{21}^9  +36 \alpha _{12}^2 \alpha _{21}^7 +36 \alpha _{20} \alpha _{21}^7 +28 \alpha _{11} \alpha _{21}^6 \\ & \qquad \qquad +56 \alpha _{12}^3 \alpha _{21}^5+168 \alpha _{12} \alpha _{20} \alpha _{21}^5 +105 \alpha _{11} \alpha _{12} \alpha _{21}^4+35 \alpha _{12}^4 \alpha
_{21}^3+70 \alpha _{20}^2 \alpha _{21}^3 +20 \alpha _{10} \alpha _{21}^3 \\ & \qquad \qquad +210 \alpha_{12}^2 \alpha _{20} \alpha_{21}^3+90 \alpha_{11} \alpha_{12}^2 \alpha_{21}^2+60 \alpha_{11} \alpha_{20} \alpha _{21}^2 +6 \alpha_{12}^5 \alpha_{21} +10 \alpha _{11}^2 \alpha_{21} \\ & \qquad \qquad +60 \alpha _{12} \alpha _{20}^2 \alpha_{21}+20 \alpha_{10} \alpha_{12} \alpha_{21}  +60\alpha_{12}^3 \alpha_{20} \alpha_{21}+10 \alpha_{11} \alpha _{12}^3+20 \alpha _{11} \alpha_{12} \alpha_{20} \\ & \qquad \qquad \left. +2 \alpha_{01} \left(5 \alpha _{21}^2+2 \alpha _{12}\right)+\alpha_{02} \left(4\alpha_{11}+\alpha_{21} \left(21 \alpha _{21}^4+60 \alpha_{12} \alpha _{21}^2+30 \alpha_{12}^2+20 \alpha _{20}\right)\right), 0 \right).
\end{align*}
After the $16$th blow-up, the system of equations takes the following form:
\begin{equation}
\label{system16}
\begin{aligned}
u_{16}' &= \frac{-3600 + p_{16,1}(z,u_{16},v_{16})}{u_{16}^4 v_{16}^5 \cdot d(z,u_{16},v_{16})^2},  \\
v_{16}' &= \frac{-240(2\alpha_{21}'(z)^2 + 2\alpha_{21}(z)^2 \alpha_{21}''(z) + 3\alpha_{12}''(z)) + p_{16,2}(z,u_{16},v_{16})}{u_{16}^5 v_{16}^3 \cdot d(z,u_{16},v_{16})^2},  \\
U_{16}' &= \frac{240(2\alpha_{21}'(z)^2 + 2\alpha_{21}(z)^2 \alpha_{21}''(z) + 3\alpha_{12}''(z)) + P_{16,1}(z,U_{16},V_{16})}{V_{16}^5 \cdot D(z,U_{16},V_{16})^2}, \\
V_{16}' &= \frac{-3600 + P_{16,2}(z,U_{16},V_{16})}{V_{16}^4 \cdot D(z,U_{16},V_{16})^2},
\end{aligned}
\end{equation}
where $p_{16,i}$ and $P_{16,i}$, $i=1,2$, are polynomials in $u_{16},v_{16}$ and $U_{16},V_{16}$, respectively, such that $p_{16,i}(z,0,v_{16}) = 0 = P_{16,i}(z,U_{16},0)$ on the exceptional curve $\mathcal{L}_{16}: \{u_{16}=0\} \cup \{V_{16}=0\}$. The polynomial expressions $d(z,u_{16},v_{16})$ and $D(z,U_{16},V_{16})$, whose zero sets are the proper transforms in these coordinates of the exceptional curves $\mathcal{L}_i$, $i=1,\dots,15$, of all previous blow-ups, satisfy $d(z,0,v_{16}) = 60 = D(z,U_{16},0)$ on the curve $\mathcal{L}_{16}$. Due to their lengthy nature we omit writing down the full expressions.
Thus, unless the condition 
\begin{equation}
\label{M3N4cond}
2\alpha_{21}'(z)^2 + 2 \alpha_{21}(z) \alpha_{21}''(z) + 3 \alpha_{12}''(z) = (\alpha_{21}^2 + 3\alpha_{12})'' \equiv 0
\end{equation}
is satisfied, we can see that the solutions of system (\ref{system16}) admit logarithmic singularities. Indeed, integrating the fourth equation in system (\ref{system16}) gives the leading order behaviour $V_{16} \sim (z-z_0)^{1/5}$. Inserting this into the third equation, $U_{16}$ has a logarithmic branch point.  Thus, for the absence of logarithmic singularities we require condition (\ref{M3N4cond}) to be satisfied, which amounts to the function $\alpha_{21}(z)^2 + 3\alpha_{12}(z)$ being at most linear in $z$. In this case, one factor of $u_{16}$ and $V_{16}$ cancel in the second resp.\ third equation of system (\ref{system16}), and by interchanging the role of dependent and independent variables, we can write the system of equations in the form
\begin{equation*}
\begin{aligned}
\frac{d z}{d V_{16}} & = \frac{V_{16}^4 \cdot D(z,U_{16},V_{16})^2}{-3600 + P_{16,2}(z,U_{16},V_{16})}, \\
\frac{d U_{16}}{d V_{16}} & = \frac{\tilde{P}_{16,1}(z,U_{16},V_{16})}{-3600 + P_{16,2}(z,U_{16},V_{16})},
\end{aligned}
\end{equation*}
which, for initial data $(z,U_{16},V_{16}) = (z_0,h,0)$ on the exceptional curve $\mathcal{L}_{16}$ from the last blow-up, becomes a regular initial value problem with analytic solutions
\begin{equation*}
z(V_{16}) = z_0  -\frac{1}{5} V_{16}^5 + O(V_{16}^6), \quad U_{16}(V_{16}) = h + O(V_{16}).
\end{equation*}

Inverting the power series for $z-z_0$ leads to series expansions for $U_{16}$ and $V_{16}$ in $(z-z_0)^{1/5}$, which translate to $5$th-root type algebraic poles in the original variables $x,y$.

We still need to show that, after each blow-up, the exceptional line $\mathcal{L}_i$ is inaccessible for the flow of the vector field, apart from at the newly introduced base point $\mathcal{P}_{i+1}$. For this we introduce the following auxiliary function of the form (obtained in \cite{Kecker2016})
\begin{equation*}
	W = H - \beta_{2,0}(z) y^2 - \beta_{3,1}(z) \frac{y^3}{x} - \beta_{4,2}(z) \frac{y^4}{x^2} - \beta_{1,0}(z) y - \beta_{2,1}(z) \frac{y^2}{x} - \beta_{3,2}(z) \frac{y^3}{x^2}.
\end{equation*}
The functions $\beta_{kl}(z)$ can be determined using the procedure described in \cite{Kecker2016}, where they are fixed so that $W$ satisfies a first-order equation with bounded coefficients. On the other hand, we can also obtain the $\beta_{kl}(z)$ in the blow-up process itself by the requirement that after each blow-up the logarithmic derivative $\frac{W'}{W}$ remains bounded on the exceptional curves, away from the base points. Using the latter method we have found (with the condition (\ref{M3N4cond}) imposed):
\begin{equation*}
	\begin{aligned}
		\beta_{2,0} &=\frac{\alpha_{2,1}'}{6}, \\
		\beta_{3,1} &=\frac{1}{15} \left(3 \alpha_{1,2}'+2 \alpha_{2,1} \alpha_{2,1}'\right), \\
		\beta_{4,2} &=\frac{1}{60} \left(9 \alpha_{2,1} \alpha_{1,2}'+5 \alpha_{1,2} \alpha_{2,1}'+6 \alpha_{2,1}^2 \alpha_{2,1}'\right), \\
		\beta_{1,0} &= \frac{1}{30} \left(2 \alpha_{1,2} \alpha_{1,2}'+3 \alpha_{2,1}^2 \alpha_{1,2}'+10 \alpha_{2,0}'+3 \alpha_{1,2} \alpha_{2,1} \alpha_{2,1}'+2 \alpha_{2,1}^3 \alpha_{2,1}'\right), \\
		\beta_{2,1} &=\frac{1}{60} \left(30 \alpha_{1,1}'-2 \alpha_{1,2} \alpha_{2,1} \alpha_{1,2}'-3 \alpha_{2,1}^3 \alpha_{1,2}'+20 \alpha_{2,1} \alpha_{2,0}'+20 \alpha_{2,0} \alpha_{2,1}'-3 \alpha_{1,2} \alpha_{2,1}^2\alpha_{2,1}' -2 \alpha_{2,1}^4 \alpha_{2,1}'\right), \\
		\beta_{3,2} &=\frac{1}{60} \left(60 \alpha_{0,2}'+30 \alpha_{2,1} \alpha_{1,1}'-8 \alpha_{1,2}^2 \alpha_{1,2}'+24 \alpha_{2,0} \alpha_{1,2}'-14 \alpha_{1,2} \alpha_{2,1}^2 \alpha_{1,2}'-3 \alpha_{2,1}^4 \alpha_{1,2}'+20 \alpha_{1,2} \alpha_{2,0}' \right. \\ & \left. +20 \alpha_{2,1}^2 \alpha_{2,0}'+20 \alpha_{1,1} \alpha_{2,1}'-12 \alpha_{1,2}^2 \alpha_{2,1} \alpha_{2,1}'+36 \alpha_{2,0} \alpha_{2,1} \alpha_{2,1}'-11 \alpha_{1,2} \alpha_{2,1}^3 \alpha_{2,1}'-2 \alpha_{2,1}^5 \alpha_{2,1}'-10 \alpha_{2,1}'' \right).
	\end{aligned}
\end{equation*}

Denoting by $\mathcal{S}_{16}(z)$ the extended space obtained from $\mathbb{P}^2$ by blowing up the cascade of base points $\mathcal{P}_1 \leftarrow \cdots \leftarrow \mathcal{P}_{16}$, we define the analogue of the space of initial values for the system by $\mathcal{S}_{16}(z) \setminus \mathcal{I}_{15}'(z)$, where $\mathcal{I}_{15}(z) = I \cup \bigcup_{i=1}^{15} \mathcal{L}_i$. In each point of this space, the system is either regular or defines solutions with a $5$th-root type singularity, whereas the exceptional curves $\mathcal{L}_i$, $i=1,\dots15$, remain inaccessible. Using similar arguments as in Proposition \ref{degree4prop}, we can thus show:
\begin{prop}
Under the condition $\alpha_{21}^2(z) + 3\alpha_{12}(z) = az + b$, $a,b \in \mathbb{C}$, the Hamiltonian system derived from (\ref{M3N4Hamiltonian}) has the quasi-Painlev\'e property, i.e.\ all movable singularities obtained by analytic continuation along finite-length curves are $5$th-root type algebraic poles.
\end{prop}

\section{Discussion}
For the examples of second-order equations in Sections \ref{sec:degree4} and \ref{sec:degree5}, as well as the Hamiltonian systems in Section \ref{sec:HamiltonianSystems} we have constructed, under the conditions by which these systems do not admit logarithmic singularities, the analogue of the space of initial values in the sense of Okamoto's space for the Painlev\'e equations. In this case, the solutions are transversal to the exceptional curve introduced by the last blow-up for any cascade of base points. The difference to the Painlev\'e case is that, in order to obtain regular initial value problems in the coordinates of the extended phase space, an additional change of dependent and independent variable is needed. The existence of these regular systems allows us to conclude, using Lemma \ref{log_bounded} with an appropriate auxiliary function and together with Painlev\'e's lemma (Lemma \ref{Painlemma}), that the only movable singularities that can occur in these equations, by analytic continuation along finite-length paths, are algebraic poles. 

This procedure thus firstly serves as an algorithm to determine, for a given second-order equation or system of two equations, what types of singularities their solutions can develop and give conditions under which there are no logarithmic singularities. In the latter case, the construction of the space of initial values allows us to show that these equations have the quasi-Painlev\'e property.

In the examples considered in this article, it is crucial that the cascades of blow-ups required to resolve the base points terminate. By a powerful theorem by Hironaka \cite{Hironaka1964} for singular algebraic varieties, the singularities of an arbitrary algebraic variety can always be resolved by a finite number of blow-ups. This is not the case, however, for flows of vector fields. An example where the sequence of blow-ups does not terminate is given by Smith's equation $y'' + 4y^3 y' + y = 0$, which is not of Hamiltonian form. This was noted by the authors in \cite{KeckerFilipuk} and is a hint that for this equation more complicated movable singularities exist than considered in this article. In fact, as noted earlier, Smith himself showed that there do exist singularities besides the algebraic poles (\ref{smithexpansion}), which are are known to be accumulation points of such algebraic poles. 
It would be an important step to find (necessary and / or sufficient) conditions for a differential equation to decide whether such behaviour is possible or not. We believe that, at least for the second-order equations in \cite{filipukhalburd1} and the Hamiltonian systems in \cite{Kecker2016} such behaviour is not possible, although we cannot show this in general.

In the Hamiltonian setting, the level sets $H(z,x,y) = c$ define, for generic $z$ and analytic functions $\alpha_{i,j}(z)$ in the Hamiltonian $H(z,x,y;\alpha_{ij})$, algebraic curves in $\mathbb{P}^2$. For the Painlev\'e Hamiltonians, and also the system with cubic Hamiltonian in section \ref{sec:cubicHam}, these level sets have genus $g=1$, i.e. represent elliptic curves. This is also expressed in the fact that the Painlev\'e transcendents, in general, are asymptotic to elliptic functions in certain sectors of the complex plane.
For the Hamiltonians of the second-order equations in sections \ref{sec:degree4} and \ref{sec:degree5}, the level sets $H(z,x,y) = c$ are algebraic curves of hyper-elliptic type, with genus $g=2$, as is the case for the other examples of Hamiltonians considered in this article. The number of blow-ups required, ranging from $14$ to $16$, is substantially larger than $9$ in the Painlev\'e case. We would like to propose several questions which will require further investigation.

Can one predict, from the form of the Hamiltonian $H(z,x,y)$, how many blow-ups will be required to completely resolve all base points? In particular, can one give conditions under which the cascades of blow-ups terminate?

Can a classification, similar to Sakai's classification \cite{Sakai2001} for the Painlev\'e equations in terms of point configurations on rational surfaces, be given for Hamiltonians defining algebraic curve of genus $g \geq 2$, and, do there exist difference equations with a similar meaning as the discrete Painlev\'e equations?

\section*{Declarations}

\paragraph{\bf Funding}
GF acknowledges the support of the National Science Center (Poland) through the grant OPUS 2017/25/B/BST1/00931. TK acknowledges support of the London Mathematical Society (LMS) and the Faculty of Mathematics, Informatics and Mechanics at the University of Warsaw (MIMUW) for travel grants to visit Warsaw in the years 2014, 2015 and 2016; these visits, were this research was initiated, were essential for the success of the project. \\

\paragraph{\bf Data availability}
Data sharing is not applicable to this article as no datasets were generated or analysed during the current study. \\

\paragraph{\bf Conflict of interest} 
The authors declare that there are no conflicts of interest.

\bibliographystyle{plain}

\vspace{2cm}
\noindent Authors' details: \\

\noindent Dr.\ Thomas Kecker \\ School of Mathematics and Physics \\
University of Portsmouth, UK \\
Email: thomas.kecker@port.ac.uk \\

\noindent Dr.\ hab.\ Galina Filipuk \\
Faculty of Mathematics, Informatics and Mechanics \\
University of Warsaw, Poland \\
Email: filipuk@mimuw.edu.pl

\end{document}